\documentclass[11pt]{amsart}

\usepackage{a4wide}
\usepackage{enumerate}
\usepackage{amsmath, amsthm, amsfonts, amssymb,accents,a4}
\usepackage{dsfont}
\usepackage{srcltx, color}

\theoremstyle{plain}
\newtheorem{theorem}{Theorem}[section]
\newtheorem{lemma}[theorem]{Lemma}
\newtheorem{corollary}[theorem]{Corollary}
\theoremstyle{remark}
\newtheorem{remark}[theorem]{Remark}

\numberwithin{equation}{section}

\def\a{\alpha}
\def\b{\beta}
\def\om{\omega}
\def\e{\varepsilon}
\def\g{\gamma}
\def\G{\Gamma}
\def\l{\lambda}
\def\p{\partial}
\def\D{\Delta}
\def\Om{\Omega}
\def\z{\zeta}

\def\per{{\text{per}}}

\def\k{\varkappa}
\def\E{\mbox{\rm e}}

\def\d{\delta}
\def\L{\Lambda}

\def\Odr{\mathcal{O}}

\def\di{\,d}
\def\iu{\mathrm{i}}

\def\I{\mathcal{I}}

\newcommand{\CC}{\mathds{C}}

\newcommand{\EE}{\mathbb{E}}

\newcommand{\NN}{\mathds{N}}

\newcommand{\PP}{\mathbb{P}}

\newcommand{\RR}{\mathds{R}}
\newcommand{\ZZ}{\mathds{Z}}

\def\pL{\mathcal{L}}
\def\Op{\mathcal{H}}
\def\S{\mathcal{S}}
\def\Dom{\mathfrak{D}}
\def\Rs{\mathcal{R}}
\DeclareMathOperator{\supp}{supp}

\def\Ho{\mathring{H}}

\def\fB{\mathcal{F}}

\def\bc{\mathcal{B}}

\def\tU{\widetilde{U}}

\def\tOp{\widetilde{\Op}}
\def\spec{\sigma}

\def\cT{\mathcal{T}}
\def\cP{\mathcal{P}}

\renewcommand\epsilon{\varepsilon}

\DeclareMathOperator{\Div}{div}
\DeclareMathOperator{\dist}{dist}

\begin{document}

\allowdisplaybreaks

\title[Weak random abstract perturbation]{Quantum Hamiltonians with weak random abstract perturbation. I. Initial length scale estimate}
\author[Borisov, Golovina, and Veseli\'c]{Denis Borisov$^{1,2}$, Anastasia Golovina$^{3}$, and  Ivan Veseli\'c$^{4}$}
\keywords{random Hamiltonian, weak disorder, random geometry, quantum waveguide, low-lying spectrum, asymptotic analysis, Anderson localization}
\subjclass[2000]{35P15, 35C20, 60H25, 82B44}
\thanks{\jobname.tex}

\maketitle

\begin{quote}
\begin{itemize}
\small
\item[1)] \emph{Department of Differential Equations, Institute of Mathematics with Computer Center, Ufa Scientific Center, Russian Academy of Sciences, Chernyshevsky. st.~112, Ufa, 450008,
Russia}

\item[2)] \emph{Department of Physics and Mathematics, Bashkir
State Pedagogical University, October rev. st.~3a, Ufa, 450000,
Russia}

\item[3)] \emph{Department of Fundamental Sciences, Bauman Moscow State Technical University, 105005, Rubtsovskaya quai 2/18,  Moscow,
Russia}

\item[4)] \emph{Department of Mathematics, Technische Universit\"at Chemnitz, 09107 Chemnitz, Germany}

\end{itemize}
\end{quote}

\begin{abstract}
We study random Hamiltonians on finite-size cubes and waveguide segments of increasing diameter.
The number of random parameters determining the operator is proportional
to the  volume of the cube. In the asymptotic regime where the cube size,
and consequently the number of parameters as well, tends to infinity,
we derive deterministic and probabilistic variational bounds on the lowest eigenvalue,
i.~e. the spectral minimum, as well as exponential off-diagonal decay of the Green function at energies
above, but close to the overall spectral bottom.
\end{abstract}

\section{Introduction}
Quantum disordered systems often exhibit localization, i.e.~the absence of propagation of wavepackets.
For random ergodic Schr\"odinger operators in $L_2(\RR^n)$ this has been established in various regions in the
energy $\times$ disorder diagram. For such models, localization comes about thanks to the local effect of random
variables (encoding the disorder in the Hamiltonian) and a global conspiracy of randomness over large scales.
A natural approach to study, and actually, prove localization, is to analyze first the spectral effects of a single random variable on a specific type of random operator, then the cumulative effect of many variables on finite but large cubes in configuration space, and finally conclude that a quantitative form of localization persists if one takes the macroscopic limit.

We take a reverse, conceptual and abstract approach. We want to formulate criteria on the properties of local perturbations (single site potentials for usual random Schr\"odinger operators) which ensure that localization will ensue in an appropriate disorder/energy regime. To illustrate what we mean, let us consider the very first result on localization in $L_2(\RR^n)$ obtained by Holden and Martinelli in \cite{MartinelliH-84}. They consider the random Schr\"odinger operator $H_\omega= -\D+ \sum\limits_{k \in \ZZ^n} \omega_k \, u(\cdot-k)$ in $L_2(\RR^n)$, where $\omega_k, k \in \ZZ^n$, are uniformly distributed on $[0,1]$ and $u(x)=\chi_{[0,1]^n}(x)$ is the characteristic function of the unit cube. The global strategy employed in \cite{MartinelliH-84} to prove localization is the multiscale analysis of Fr\"ohlich and Spencer.
On the local level, the properties of the single site perturbation $\omega_k\mapsto  \omega_k \, u(\cdot-k)$ are essential. It is linear in $\omega_k$, nonnegative on $L_2(\RR^n)$ and strictly positive on $L_2([0,1]^n)$. The question we raise is: If the function $u=\chi_{[0,1]^n}$ is replaced by a more general function, or even an operator distinct from an multiplication operator, which properties should it have in order to ensure pure point spectrum of $H_\omega$? While this has been studied for a number of specific, physically relevant, models, our approach is conceptual. We want to understand a set of sufficient conditions on the local building blocks of the Hamiltonian (single site perturbations) which ensures localization, or at least important partial results used on the road to localization. Although in this paper we do not provide a complete answer to the above question, we make a first important step. Namely, we provide an initial length scale estimate, that is one of the
 main steps in proving
spectral localization via multiscale analysis, for a very wide class of random Hamiltonians with weak disorder.

The second key ingredient to make the multiscale analysis work is a Wegner estimate.
The role of the two ingredients is the following:
the multiscale analysis is a induction procedure over a sequence of increasing
length scales. While the initial length scale estimate provides
induction anchor, the Wegner estimate guarantees that the induction step works.
Physically, the Wegner estimate  ensures that resonances between spectra of
disjoint subsystems occur only with small probability. In a sequel paper
we plan to give a set of conditions on general, abstract  random Hamiltonians
which imply the Wegner estimate. This set is distinct, but similar to the
conditions we impose in the present paper. Thus for random Hamiltonians which
satisfy both requirements localization via multiscale analysis follows.

The indication how to implement the proof of
Wegner estimate is provided by Lemma 2.3 below.
It describes the lifting of the spectral bottom for periodic
configurations of the random coupling constants.
This ensures that there is a (small) energy interval near the minimum of the spectrum
of the original, unperturbed operator which is uncovered by the random perturbations:
There exists a operator in the ensemble whose resolvent set contains
the mentioned energy interval. In this situation the vector-field method introduced in

\cite{Klopp-95} by Klopp and developed in \cite{HislopK-02} and \cite{GhribiHK-07}
can be applied.

While our theorems cover a substantially more general setting, let us describe here in the introduction a special case of the model we consider:
Let $\pL_1, \pL_2 \colon  {H^2}({[0,1]^n})\to L_2({[0,1]^n})$ be bounded symmetric linear operators, $\epsilon>0$,
$\pL(t):=t \pL_1 + t^2 \pL_2$, $t\in[-\epsilon,\epsilon]$,
be an operator family,  $(\S(k)u)(y)=u(y+k), y \in \RR^n, k \in \ZZ^n$ be the shift operator, $\omega_k, k \in\ZZ^n$, be a sequence of numbers with values in $[-1,1]$,
and
\begin{equation*}
\Op^\e(\omega):=-\D + \pL^\e(\omega),\quad \pL^\e(\omega):=\sum\limits_{k\in\G} \S(k) \pL(\e \omega_k) \S(-k),
\end{equation*}
We have no specific requirements on the type of 
operators $\pL_1,\pL_2$: they could be multiplication, differential,
or integral operators or a combination of these. In Section \ref{Examples} we give a number of examples covered by our general setting, including scalar potentials,
magnetic fields, random metrics, Laplacians on infinite strips and layers with random boundary, as well as integral operators.

Our two assumptions on the single site operators $\pL_1,\pL_2$ are the following:
Let $-\D$ be the negative Neumann Laplacian on ${[0,1]^n}$, $\mathds{1}\colon {[0,1]^n} \to 1 $ the constant function  and
$u$ the unique solution to $-\D u=\pL_1\mathds{1}$ with $\int\limits_{[0,1]^n} u \, dy=0$.  We assume that

(A1')\  $\int\limits_{[0,1]^n} \pL_1 \mathds{1}dy=0$ \qquad and \qquad
(A2')\  $\int\limits_{[0,1]^n}\pL_2\mathds{1} dy -
\int\limits_{[0,1]^n} u \overline{\pL_1\mathds{1}} dy >0$.

%

Our first result is that the lowest eigenvalue $\l_{N}^\e(\omega)$ of the restriction
$\Op^\e_{N}(\omega)$ of $\Op^\e(\omega)$ to
$\Pi_{N}:=\big\{y\in \RR^n: \, y=\sum\limits_{i=1}^{n} a_i e_i,\, a_i\in(0,N)\big\}$
with Neumann boundary conditions obeys
\begin{equation*}
\l_{N}^\e(\omega)\geqslant c_{2} \e^2 \sum\limits_{k\in\G_{N}} \omega_k^2, \quad \G_{N}:= \ZZ^n \cap [0,N)^n
\end{equation*}
provided $N>N_{1}, \text{ and } 0<\e<\frac{c_{1}}{N^2}$.
Here $c_{1}$, $c_{2}$, $N_{1}\in(0,\infty)$ are independent of $\e$ and $N$.

If $\omega_k, k\in \ZZ^d$ form an i.i.d.~sequence of random variables, we deduce that for small, but not too small values of $\epsilon>0$
\begin{equation*}
\PP\left(\omega\in\Om:\, \l_{\a,N}^\e\leqslant N^{-\frac{1}{2}}\right)\leqslant N^{n\left(1-\frac{1}{\g}\right)}\E^{-c_{4}N^{\frac{n}{\g}}}
\end{equation*}
where
constant $c_{4}>0$ depends only on the distribution of $\omega_0$. Finally, we prove a Combes-Thomas bound for the general class of operators we consider, and derive a initial scale estimate, as it is used for the induction anchor of the multiscale analysis.  A Combes-Thomas estimate is a bound on the off-diagonal exponential decay of the integral kernel of $(\Op^\e_{N}(\omega)-E)^{-1} $ provided
an a-priori lower bound on $\dist(E,\sigma(\Op^\e_{N}(\omega)))$ is known. Here $\spec(\cdot)$ denotes the spectrum of an operator.
The novelty of our result is that the considered operators need only be block diagonal with respect to the decomposition $\bigoplus\limits_{k\in\ZZ^n} L_2\big([0,1]^n+k\big)$,
but not necessarily a differential operator.

In Section \ref{ss:whole.space} we show that the model described in this introduction is
covered by the more general, abstract model defined in Section \ref{s:results}.

\subsection*{History and earlier results}
\label{s:History}
\subsubsection*{Results on random waveguides}
The results presented here are a generalization and improvement of those in \cite{BorisovV-11}:
In \cite{BorisovV-11} we considered \emph{randomly wiggled} quantum waveguides in the ambient space $\RR^2$.
For this specific model a variational estimate analogous to
Theorem  \ref{th:deterministic.lower.bound}
was established in \cite[Corollary 4.2]{BorisovV-11}.
In \cite{BorisovV-13} we studied an apparently very similar disordered model, namely a \emph{randomly curved} waveguide.
For this model the hypotheses (A1) and (A2) are not satisfied and our analysis showed that the lowest eigenvalue
$\lambda^\epsilon_N(\omega)$
exhibits a behaviour distinct (in some sense \emph{opposite}) to the one encoded in inequality (\ref{3.2}).
Common to both types of random waveguides studied in \cite{BorisovV-11} and  \cite{BorisovV-13} is a
non-monotone dependence on the random variables $\omega_k$. This is a challenge to the mathematical analysis, as will be elaborated further below.
To the best of our knowledge    the first disordered model of a quantum waveguide was studied in \cite{KleespiesS-00}.
There the width of the waveguide is determined by a sequence of random parameters $\omega_k$.
This gives rise to a monotone influence of the parameters and facilitates the study of the spectrum.
In \cite{KleespiesS-00} in addition to an initial length scale estimate for the Green's function a Wegner estimate was provided,
yielding spectral localization near the bottom of the spectrum.
\subsubsection*{Weak disorder quantum Hamiltonians}
Our model depends on a global parameter $\epsilon>0$.
It tunes the overall strength of the disorder present in the operator.
The interest is now, to identify an energy interval, depending on the disorder strength $\epsilon $, where an
initial length scale estimate, a Wegner estimate and spectral localization hold. Corollary \ref{corollary:Lifschitz-tail-regime}
provides such a statement concerning the initial length scale estimate.
Most detailed results  identifying energy regimes with spectral localization in the weak disorder regime have been obtained
for the Anderson model on $\ell^2(\ZZ^n)$, or its continuum analog, the alloy type model on $L^2(\RR^n)$.
Corresponding to the general setting of this paper, we will restrict our discussion to continuum  models, i.e.~quantum Hamiltonians
defined on $\RR^n$ or an open, unbounded subset thereof.
In \cite{Klopp-02c} it is proven that under the assumption
\begin{equation*}
(\neg A1) \hspace*{\fill} \hspace*{3cm}  u \in L_c^\infty(\RR^n), \qquad \int u(x) \, dx \neq 0
\end{equation*}
there is an energy interval $J_\epsilon\in \RR$ with size of the order $\epsilon$  such that the alloy type model
\begin{equation*}
-\Delta + V_{\per} +\epsilon \sum_{k\in \ZZ^n} u(\cdot -k)
\end{equation*}
exhibits spectral and dynamical localization in $J_\epsilon$ almost surely.
Here $V_{\per}$ is a bounded $\ZZ^n$-periodic potential.
Due to assumption $(\neg A1)$, \cite{Klopp-02c} does not cover the situation considered here (if we assume that the random perturbation
is a multiplication operator). In this sense our result, when specialised to the case that the random part of the Hamiltonian is a
potential, complements the result of  \cite{Klopp-02c}.
\subsubsection*{Spectral analysis of non-monotone random Hamiltonians}
The proofs of initial length scale estimates and Wegner estimates simplify greatly
if the random variables $\omega_k$ influence the quadratic form associated to the Hamiltonian in a monotone way.
If this monotonicity property is violated one has to identify and use specific properties of the model at hand  in order to replace monotonicity.
This has been carried out for alloy type models of changing sign e.g.~in
\cite{Klopp-95a,Stolz-00,Klopp-02c,Veselic-02a,HislopK-02,KostrykinV-06, KloppN-09a,LeonhardtPTV},
for random displacement models e.g.~in
\cite{Klopp-93,BakerLS-08,GhribiK-10,KloppLNS-12},
for random magnetic fields e.g.~in
\cite{Ueki-94,Ueki-00,HislopK-02,KloppNNN-03,Ueki-08, Bourgain-09}, \cite{ErdoesH-12a,ErdoesH-12c},
and Laplace-Beltrami operators with random metrics  e.g.~in\cite{LenzPV-04,LenzPPV-08,LenzPPV-09}.

\subsection*{Innovations}
We list the results, methods and conceptual innovations obtained in the paper.
\begin{itemize}
\item
 We establish a variational lower bound for the spectral minimum of random Hamiltonians an arbitrary large, finite boxes $\Pi_N$.
As the box size $N$ grows, the number of random variables influencing   the random Hamiltonian $\Op^\e_N(\omega)$  grows as well, namely proportional to the volume of $\Pi_N$. Thus the variational problem involves a large (and increasing) number or parameters.
\item
The basic assumption on the influence of the individual parameters on the random Hamiltonian $\Op^\e_N(\omega)$ is the validity of a certain Taylor formula,
cf.~(\ref{2.1}), as well as Assumptions (A1') \& (A2'), or their generalizations (A1) \& (A2).
In contrast to the standard approach, we do not require the dependence to be linear (not even rational).
\item
The variational bounds are proven using a perturbative framework based on a
non-self-adjoint modification of Birman-Schwinger principle proposed and developed in
\cite{Gadylshin-02}, see also \cite{Borisov-06}, \cite{BorisovG-08}, and \cite{Borisov-11}.
This abstract approach allows a uniform treatment of many types of random Hamiltionians studied before
(random scalar potentials, random magnetic fields, randomly perturbed quantum waveguides), as well as new types (e.g.~integral operators, randomly perturbed quantum layers).
\item
We establish a general Combes-Thomas estimate. It does not require the Hamiltonian to be a differential operator, rather it could contain an integral operator part as well. To the best of knowledge of the authors such estimates have been so far obtained only for local operators.
\item
Now two probabilistic results follow:
First we establish an upper bound on the probability of finding an eigenvalue of $\Op^\e_{N}(\omega)$ very close to $\Lambda_0$.
The probability is exponentially small in the size $N$, while the notion of 'very close'  depends on $N$ as well.
With the same probability we establish that the Greens function for energies above, but close to $\Lambda_0$ of the resolvent of $\Op^\e_{N}(\omega)$
decays exponentially in space.
\item
The size of the energy interval above $\Lambda_0$ can be expressed as a function of the weak coupling parameter $\epsilon$,  instead as of $N$.
This implies an estimate on the 'Lifschitz tail regime' quantified in terms of a small disorder parameter $\epsilon$.
The size of the interval is not quite,
but pretty close to quadratic in $\epsilon$. A quadratic behaviour is the best one could expect.
\end{itemize}

\section{Formulation of problem and main results}
\label{s:results}

Let $x'=(x_1,\ldots,x_n)$, $x=(x',x_{n+1})$ be Cartesian coordinates in $\RR^n$ and $\RR^{n+1}$, respectively, where $n\geqslant 1$. By $\Pi$ we denote the multidimensional layer $\Pi:=\{x:\, 0<x_{n+1}<d\}$
of width $d>0$. In space $\RR^n$ we introduce a periodic lattice $\G$ with a basis $e_1$, \ldots, $e_n$; the unit cell
of this lattice is denoted
by $\square'$, i.e., $\square':=\{x': x'=\sum\limits_{i=1}^{n} a_i e_i,\; a_i\in(0,1)\}$. We denote $\square:=\square'\times(0,d)$.

For some $t_0>0$  we denote by $\pL(t)$, $t\in[-t_0,t_0]$, a family of linear operators from ${H^2}(\square)$ into $L_2(\square)$ given by
\begin{equation}\label{2.1}
\pL(t):=t \pL_1 + t^2 \pL_2 + t^3 \pL_3(t),
\end{equation}
where $\pL_i: {H^2}(\square)\to L_2(\square)$ are bounded symmetric linear operators and $\pL_3(t)$ is bounded uniformly in $t\in[-t_0,t_0]$.

Given $u\in {H^2}(\Pi)$,
it is clear that $u\in {H^2}(\square)$ and function $\pL_i u$ is thus well-defined as an element of $L_2(\square)$. Now we can extend the function $\pL_i u$ by zero in $\Pi\setminus\square$ and this extension is an element of $L_2(\Pi)$. In the sense of the above continuation, in what follows, we regard the operators $\pL_i$ as acting from ${H^2}(\Pi)$ into $L_2(\Pi)$. We stress that, in general, the operators $\pL_i$ are unbounded as operators in $L_2(\Pi)$.

The main object of our study is the operator
\begin{equation}\label{2.2}
\Op^\e(\omega):=-\D +V_0+ \pL^\e(\omega),\quad \pL^\e(\omega):=\sum\limits_{k\in\G} \S(k) \pL(\e \omega_k) \S(-k),
\end{equation}
in $\Pi$. Here  $\e$ is a small positive parameter,
$\omega_k, k \in\ZZ^n$ a sequence of numbers with values in $[-1,1]$, $V_0(x)=V_0(x_{n+1})$ is a measurable bounded potential depending only on the transversal variable $x_{n+1}$,
$\S(k)$ stands for the shift operator: $(\S(k)u)(x)=u(x'+k,x_{n+1})$.
The boundary condition on $\p\Pi$ is either of Dirichlet or Neumann type.
We denote this condition by
\begin{equation}\label{2.8}
\bc u=0
\end{equation}
on $\p\Pi$, and $\bc u=u$ or $\bc u=\frac{\p u}{\p x_{n+1}}$.
We consider also the situations when on the upper and lower boundaries of $\p\Pi$ we have different boundary conditions.
Say, on the upper boundary we have Dirichlet condition, while on the lower boundary Neumann condition is imposed.

We consider the operator $\Op^\e$ as an unbounded one in $L_2(\Pi)$ on the domain $\Dom(\Op^\e):=\{u\in {H^2}(\Pi):\, \text{(\ref{2.8}) is satisfied on $\p\Pi$}\}$. The action of the second term in the right hand side of (\ref{2.2}) can be also understood as follows: Given $u\in {H^2}(\Pi)$, we consider the restriction of $u$ on the cell $\square_k:=\{x: x-(k,0)\in\square\}$ for each $k\in\G$. Then, identifying cells $\square_k$ and $\square$, we apply the operator $\pL(\e \omega_k)$ to $u\big|_{\square_K}$ and the result is how $\pL^\e u$ is defined on $\square_k$.

For sufficiently small $\e$ operator $\pL^\e$ is relatively bounded w.r.t. the Laplacian on $\Dom(\Op^\e)$ with relative bound smaller than one
and the latter operator is self-adjoint.
Hence, by the Kato-Rellich theorem  operator $\Op^\e$ is self-adjoint, as well.

Our results concern operators on large, finite pieces
\begin{equation}\label{2.3}
\Pi_{\a,N}:=\Big\{x:\  x'=\a+\sum\limits_{i=1}^{n}a_i e_i,\ a_i\in(0,N),\ 0<x_{n+1}<d\Big\},
\end{equation}
of the layer $\Pi$, where $\a\in\G$ and $N\in\NN$ are arbitrary. We let
\begin{equation*}
\G_{\a,N}:=\Big\{x'\in\G:\  x'=\a+\sum\limits_{i=1}^{n}a_i e_i,\ a_i=0,1,\ldots,N-1\Big\}
\end{equation*}
and observe that $\Pi_{\a,N}=\bigcup\limits_{k\in\G_{\a,N}} \square_k$.

We introduce the operator
\begin{equation*}
\Op^\e_{\a,N}(\omega):=-\D+V_0+\pL^\e_{\a,N}(\omega),\quad \pL^\e_{\a,N}(\omega):=\sum\limits_{k\in\G_{\a,N}} \S(k) \pL(\e \omega_k) \S(-k)
\end{equation*}
in $L_2(\Pi_{\a,N})$ subject to boundary condition (\ref{2.8}) on $\g_{\a,N}:=\p\Pi_{\a,N}\cap \p\Pi$ and to Neumann condition on $\p\Pi_{\a,N}\setminus\overline{\g_{\a,N}}$. The domain of $\Op^\e_{\a,N}(\omega)$ is
\begin{align} \label{operator.domain}
\Dom(\Op^\e_{\a,N}):=\big\{u\in {H^2}(\Pi_{\a,N}): \ &\text{$u$ satisfies (\ref{2.8}) on $\g_{\a,N}$}
 \\
 &\text{and Neumann condition on $\p\Pi_{\a,N}\setminus\overline{\g_{\a,N}}$}\big\}. \nonumber
\end{align}
The reason why we impose Neumann boundary conditions is the following: We want to give lower bounds on the first eigenvalue of finite volume Hamiltionians
$\Op^\e_{\a,N}(\omega)$. Since Neumann conditions produce the lowest ground state energy, this covers the `worst case scenario'.

By $\L_0$ we denote the lowest eigenvalue of the operator
\begin{equation*}
-\frac{d^2}{dx_{n+1}^2}+V_0\quad \text{on}\quad (0,d)
\end{equation*}
subject to boundary condition (\ref{2.8}). The associated eigenfunction normalized in $L_2(0,d)$ is denoted by
$\psi_0: (0,d) \to \RR$.
Let $\Op_\square$ be the Schr\"odinger operator $-\D+V_0$ in $\square$ subject to boundary condition (\ref{2.8}) on $\p\square\cap\p\Pi$ and to Neumann condition on $\p\square\setminus\p\Pi$.
Note that $\Op_\square\psi_0=\Lambda_0\psi_0$, where here $\psi_0: \square\to\RR$ is given by the longitudinally constant extension $\psi_0(x)=\psi_0(x_{n+1})$.
The second-lowest eigenvalue of $\Op_\square$ is denoted by $\Lambda_1$.

We make the following assumptions for operators $\pL_i$.

\begin{enumerate}\def\theenumi{A\arabic{enumi}}

\item\label{as1} The identity
\begin{equation*}
(\pL_1\psi_0,\psi_0)_{L_2(\square)}=0
\end{equation*}
holds true.

\item\label{as2} Let $U$ be the solution to the equation
\begin{equation}\label{2.5}
(\Op_\square-\L_0)U=\pL_1\psi_0,
\end{equation}
and orthogonal to $\psi_0$ in $L_2(\square)$. We assume that
\begin{equation}\label{2.6}
c_{0}:=(\pL_2\psi_0,\psi_0)_{L_2(\square)} -(U,\pL_1\psi_0)_{L_2(\square)}>0.
\end{equation}
\end{enumerate}
The two conditions on $U$ in Assumption~(\ref{as2}) are uniquely solvable since by Assumption~(\ref{as1}) $\pL_1\psi_0$ is orthogonal to $\psi_0$ in $L_2(\square)$.

By $\l_{\a,N}^\e$ we denote the smallest eigenvalue of $\Op_{\a,N}^\e$.
Our first result reads as follows:

\begin{theorem}
\label{th:deterministic.lower.bound}
There exist positive constants $c_{1}$, $c_{2}$, $N_{1}$
such that for
\begin{equation}\label{eq:epsilon-N-relation}
N\geqslant N_1,\quad \text{ and }\quad 0<\e<\frac{c_{1}}{N^4}
\end{equation}
the estimate
\begin{equation*}
\l_{\a,N}^\e(\omega)-\L_0\geqslant \frac{c_{2} \e^2 }{N^n} \sum\limits_{k\in\G_{\a,N}} \omega_k^2
\end{equation*}
holds true. In particular, the minimum of $\l_{\a,N}^\e$ w.r.t. $\omega_k$ is $\L_0$ and it is achieved as $\omega_k=0$, $k\in\G_{\a,N}$.
\end{theorem}


\begin{remark}
A very simple interpretation of Assumptions~(\ref{as1}),~(\ref{as2}) is as follows. They are equivalent to the condition
\begin{equation}\label{2.9}
\text{ There exists } C=\mathrm{const}>0 \quad : \quad  \L^t-\L_0\geqslant Ct^2,\quad  \text{ for } t\in \RR,\quad |t|\ \text{small }.
\end{equation}
for the lowest eigenvalue $\L^t$ of the operator $\Op_\square+\pL(t)$. The reason is that the three-term asymptotics for $\L^t$ reads as
\begin{equation*}
\L^t=\L_0+t (\pL_1\psi_0,\psi_0)_{L_2(\square)}+t^2 \big(
(\pL_2\psi_0,\psi_0)_{L_2(\square)} -(U,\pL_1\psi_0)_{L_2(\square)}\big) + \Odr(t^3).
\end{equation*}
Hence Assumptions~(\ref{as1}),~(\ref{as2}) are equivalent to (\ref{2.9}).

Inequality (\ref{2.9}) yields that the minimum of $\L^t$ w.r.t. $t$ is $\L_0$ and it is achieved at $t=0$. In Theorem~\ref{th3.1} (see also Lemma~\ref{lm3.2}) we prove the same for $\l^t_{\a,N}$, i.e., $\l^t_{\a,N}$ is minimal as the perturbation is absent. And this happens mostly thanks to Assumption~(\ref{as1}).
There are similar but distinct models, where minimizing the ground state eigenvalue  corresponds not to the minimal (i.e.~absent) perturbation, but to the maximal one,
cf.~\cite{BorisovV-13}.

We observe that in order to satisfy Assumption~(\ref{as2}), the scalar product $(\pL_2\psi_0,\psi_0)_{L_2(\square)}$ must be positive. The reason is that $(U,\pL_1\psi_0)_{L_2(\square)}\geqslant 0$. Indeed, integrating by parts and applying the minimax principle, it is easy to see that
\begin{equation*}
(U,\pL_1\psi_0)_{L_2(\square)}=\|\nabla U\|_{L_2(\square)}^2+(V_0U,U)_{L_2(\square)}-
\L_0\|U\|_{L_2(\square)}^2\geqslant (\L_1-\L_0)\|U\|_{L_2(\square)}^2,
\end{equation*}
where $\L_1$ is the second lowest eigenvalue of $\Op_\square$,
since $U$ is orthogonal to $\psi_0$.
This inequality provides also an upper bound for $(U,\pL_1\psi_0)_{L_2(\square)}$. First it implies
\begin{equation*}
\|U\|_{L_2(\square)}\leqslant \frac{1}{\L_1-\L_0} \|\pL_1\psi_0\|_{L_2(\square)}
\end{equation*}
and thus,
\begin{equation*}
\big|(U,\pL_1\psi_0)_{L_2(\square)}\big|\leqslant \frac{1}{\L_1-\L_0}\|\pL_1\psi_0\|_{L_2(\square)}^2.
\end{equation*}
Then a \emph{sufficient condition} ensuring (\ref{2.6}) is
\begin{equation*}
(\pL_2\psi_0,\psi_0)_{L_2(\square)} > \frac{1}{\L_1-\L_0} \|\pL_1\psi_0\|_{L_2(\square)}^2.
\end{equation*}
\end{remark}

Property (\ref{2.9}) implies an estimate on the spectral minimum of the operator $\Op^\e(\omega)$ on the infinite domain
for periodic configurations $\omega\in \Omega$.

\begin{lemma}
\label{l:lifting.introduction}
Consider the particular configuration $\tilde \omega \in \Omega$ with $\tilde \omega_k=1$ for all $k \in \G$.
Then there exists $\rho\in(0,\infty)$ independent of $\epsilon, \alpha, N$, such that
\begin{equation*}
\forall \, \a \in \G, N\in \NN, \e \geqslant 0: \quad
\l_{\a,N}^\e(\tilde \omega) \geqslant  \L_0+c_{0}\e^2- \rho \e^3
\end{equation*}
and
\begin{equation*}
\forall \, \e \geqslant 0: \quad
\inf \sigma(\Op^\e(\tilde \omega)) \geqslant  \L_0+c_{0}\e^2- \rho \e^3.
\end{equation*}
\end{lemma}

Our last deterministic result provides a Combes-Thomas estimate for the general class of operators we consider.

\begin{theorem}Let $\a, \b_1, \b_2\in \G$, $m_1, m_2\in\NN$ be such that $B_1:=\Pi_{\b_1,m_1}\subset \Pi_{\a,N}$, $B_2:=\Pi_{\b_2,m_2}\subset \Pi_{\a,N}$. There exists $N_2\in \NN$ such that for $N\geqslant N_2$ the bound
\begin{equation}
\|\chi_{B_1}(\Op^\e_{\a,N}(\omega)-\l)^{-1}\chi_{B_2}\|_{L_2(\Pi_{\a,N})\to L_2(\Pi_{\a,N})} \leqslant \frac{C_1}{\d} \E^{-C_2\d \dist(B_1, B_2)},
\end{equation}
holds, where $C_1$, $C_2$ are positive constants independent of $\e$, $\a$, $N$, $\d$, $\b_1$, $\b_2$, $m_1$, $m_2$, $\l$ and  $\d:=\dist(\l,\spec(\Op^\e_{\a,N}(\omega)))>0$.
\end{theorem}


Now we formulate our probabilistic results, and introduce for this purpose the assumptions on the randomness.
Let $\omega:=\{\omega_k\}_{k\in\G}$ be a sequence of independent identically distributed random variables with the distribution measure $\mu$,
with support in $[-1,1]$. We assume that $b_-\leqslant 0\leqslant b_+$ and  $b_-<b_+$, where $b_-=\min \supp \mu$ and $b_+=\max \supp \mu$.
This gives rise to the product probability measure $\PP=\bigotimes_{k\in\G} \mu$ on the configuration space $\Om:=\times_{k\in\G} [-1,1]$; the elements of this space are sequences $\omega:=\{\omega_k\}_{k\in\G}$. By $\EE(\cdot)$ we denote the expectation value of a random variable w.r.t.~$\PP$.

Now we are in position to formulate our main probabilistic results.

\begin{theorem}\label{th2.1}
Let $\g\in\NN$, $\g\geqslant 17$.
Then for $N\geqslant N_1$, where $N_1$ comes from Theorem~\ref{th:deterministic.lower.bound},
%
the interval
\begin{equation*}
I_N:=\left[\frac{c_{3}}{\EE(|\omega_k|) N^{\frac{1}{4}}}, \frac{c_{1}}{N^{\frac{4}{\gamma}}}\right]
\quad c_{3}:= \frac{2}{\sqrt c_{2}}
\end{equation*}
is non-empty.
For $N\geqslant N_1$ and $\e\in I_N$, the estimate
\begin{equation*}
\PP\left(\omega\in\Om:\, \l_{\a,N}^\e-\L_0\leqslant N^{-\frac{1}{2}}\right)\leqslant N^{n\left(1-\frac{1}{\g}\right)}\E^{-c_{4}N^{\frac{n}{\g}}}
\end{equation*}
holds true. Here the constant $c_{4}>0$ depends on $\mu$ only.
\end{theorem}

Our next statement is the initial length scale estimate.

\begin{theorem}\label{th2.2}
Let $\alpha \in \Gamma$,  $\g\in\NN$, $\g\geqslant 17$,
 $N\geqslant N_1$, and $\epsilon \in I_N$.
Fix  $\b_1,\b_2\in\G_{\a,N}$, $m_1, m_2>0$ such that $B_1:=\Pi_{\b_1,m_1}\subset\Pi_{\a,N}$, $B_2:=\Pi_{\b_2,m_2}\subset\Pi_{\a,N}$. Then there exists a constant $c_{5}$ independent of $\e$, $\a$, $N$, $\b_1$, $\b_2$, $m_1$, $m_2$ such that for
$N\geqslant \max\{N_1^\gamma,K_1^\gamma,N_2\}$
\begin{align*}
\PP\left(\forall \, \lambda \leqslant \Lambda_0 + \frac{1}{2 \sqrt N}:\, \|\chi_{B_1} (\Op^\e_{\a,N}-\l)^{-1} \chi_{B_2} \| \leqslant 2\sqrt{N} \E^{-\frac{c_{5}\dist(B_1,B_2)}{\sqrt{N}}}
\right)\geqslant 1-N^{n\left(1-\frac{1}{\g}\right)} \E^{-c_{4} N^{\frac{n}{\g}}},
\end{align*}
where $\|\cdot\|$ denotes the norm of an operator in $L_2(\Pi_{\a,N})$ and $\chi_B$ stands for the characteristic function of set $B$.
\end{theorem}

\begin{corollary}\label{corollary:Lifschitz-tail-regime}
Let $\alpha \in \Gamma$, $ \gamma \in \NN, \gamma\geqslant 17$  and $c_{1}$ be as in (\ref{eq:epsilon-N-relation}).
Choose $\epsilon >0$ and $N \in \NN$, $N\geqslant \max\{N_1^\gamma,K_1^\gamma,N_2\}$, such that $N= (\epsilon/c_{1})^{-\gamma/4}$.
Let  $\b_1,\b_2$, $m_1, m_2$,  $B_1$, $B_2$, $c_{5}$ be as in Theorem  \ref{th2.2}.
Then
\begin{align*}
\PP\left(\forall \, \lambda \leqslant \Lambda_0 + \frac{1}{2} \left(\frac{\epsilon}{c_{1}}\right)^{\gamma/8}
 :\, \|\chi_{B_1} (\Op^\e_{\a,N}-\l)^{-1} \chi_{B_2} \|
 \leqslant 2  \left(\frac{\epsilon}{c_{1}}\right)^{-\gamma/8} \E^{-c_{5}\dist(B_1,B_2) \left(\frac{\epsilon}{c_{1}}\right)^{\gamma/8}}\right)
\\
\geqslant
1-   \left(\frac{\epsilon}{c_{1}}\right)^{-n(\gamma-1)/2} \E^{-c_{4}  \left(\frac{\epsilon}{c_{1}}\right)^{-\frac{n}{2}}}.
\end{align*}
\end{corollary}

Note that, since $ N^{-4/\gamma} \sim \epsilon$,  and since in applications we have $  \dist(B_1,B_2) \sim N$ we have
\[
\dist(B_1,B_2) \, \left(\frac{\epsilon}{c_{1}}\right)^{\gamma/8} \sim \epsilon^{-\gamma/8} \sim
N^{\frac{1}{2}}
\gg 1
\]
Thus we are indeed witnessing an off-diagonal exponential decay of the Green's function, with high probability.
With the smallest value  $\gamma=17$ which is allowed, we obtain an energy interval with width of order $\e^{17/8}$
which is not much smaller than $\e^2$.
Intuitively, one would expect that in the weak disorder regime (if the first order perturbation annihilates) the Lifschitz tail regime interval is of order $\e^2$.
So, in this respect our result is suboptimal. This is the price we pay for treating very general perturbations, instead of, say, just multiplication operators.

Corollary \ref{corollary:Lifschitz-tail-regime} quantifies a Lifschitz-tail regime (an energy interval) in the weak disorder regime
(a small constant multiplying the random variables).
Lifschitz tails denote the exponential thinness  of the infinite volume
integrated density of states
near the bottom of the spectrum.
Our results have nothing to say about the integrated density of states,
since in the limit $N \to \infty$ the coupling $\epsilon$ shrinks to zero.
However, when it comes to proving localization, one always uses some kind of finite volume criterion,
like the probabilistic initial length scale decay estimate for the Green's function (even is Lifschitz asymptotics of the
integrated density of states have been established).
Thus, for this purpose our estimate is equally good as establishing a \emph{Lifschitz tail regime on the energy axis}.

\begin{remark}[More general unperturbed part]
Our model admits an abstract perturbed part $\pL^\e$ while the unperturbed operator $-\Delta +V_0$ is explicit.
The above listed results remain valid if we replace $-\Delta +V_0$ by a more general operator $\pL_0$, as long as it satisfied the following list of conditions:
\begin{enumerate}[(i)]
\item
$\pL_0$ maps $H^2(\Pi)$ to $L_2(\Pi)$ and is self-adjoint on
$\Dom(\pL_0):=\big\{u\in {H^2}(\Pi): \bc u=0 \ \text{ on } \partial \Pi\}$
\item
The restriction $\pL_{0,\square}$ of $\pL_0$ to $H^2(\square)$ with
boundary condition $\bc u=0 $ on $\p\square\cap\p\Pi$ and Neumann condition on $\p\square\setminus\p\Pi$
has spectrum $\sigma(\pL_{0,\square})\subset \{\Lambda_0\} \cup [\Lambda_1, \infty)$, where $\Lambda_0 <\Lambda_1$, $\Lambda_0$ is non-degenerate
and has a normalized, a.e.~positive eigenfunction $\psi_0$ satisfying $\pL_{0,\square} \psi_0=\Lambda_0 \psi_0$.
\item
Let $\psi_0^\per$ be the periodic extension of $\psi_0$ to $\Pi$:
$\psi_0^\per(x'+k,x_{n+1})=\psi_0^\per(x',x_{n+1})$
for all $x'\in \square, x_{n+1}\in(0,d), k\in \G$, and $\psi_0^N=N^{-1/2}\psi_0^\per \, \chi_{\Pi_{N}}$ on $L_2(\Pi_{N})$.
Let $\pL_{0,\alpha,N}$ be the restriction of $\pL_{0}$ with domain (\ref{operator.domain}). Assume that
$\Lambda_0=\inf \sigma(\pL_{0,\square})= \inf \sigma(\pL_{0,\alpha, N})$ and that
$\pL_{0} \psi_0^\per=\Lambda_0 \psi_0^\per$ as well as
$\pL_{0,N} \psi_0^N=\Lambda_0 \psi_0^N$.
\item
For any $\omega\in \Omega$ and $\pL^\epsilon_{\alpha,N}(\omega):=\pL^\epsilon_{0,\alpha,N}+\pL_{\alpha,N}(\epsilon \omega) $
we have for the spectral infimum the bracketing inequality
\[
\lambda(\pL^\epsilon_{\alpha,N}(\omega)) \geqslant \min_{\beta \in M_{K,\g}} \lambda(\pL^\epsilon_{\beta,K}(\omega) )
\]
where $ K,\gamma\in \NN$, $N=K^\gamma$ and $M_{K,\g}=K\G\cap\G_{\a,N}$.
\item
With respect to the decomposition  $\bigoplus\limits_{k\in\ZZ^n} L_2\big(\square+(k,0)\big)$, $\pL_0$ is a block-diagonal operator.
\end{enumerate}
\end{remark}

\begin{remark}[More general perturbation]
Although we have assumed that operators $\pL_1$ and $\pL_2$ are independent of $t$, it is possible to treat also the case when these operators depend on $t$, i.e., $\pL_1=\pL_1(t)$, $\pL_2(t)$. In this case we should suppose that these operators considered as acting from $H^2(\square)$ into $L_2(\square)$ are bounded uniformly in $t$. The identity in Assumption~\ref{as1} should hold true uniformly in $t$, i.e.,
\begin{equation*}
(\pL_1(t)\psi_0,\psi_0)_{L_2(\square)}=0\quad\text{for each}\  t\in[-t_0,t_0].
\end{equation*}
And inequality (\ref{2.6}) should be modified as follows:
\begin{equation*}
(\pL_2(t)\psi_0,\psi_0)_{L_2(\square)} -(U(t),\pL_1(t)\psi_0)_{L_2(\square)}\geqslant c_{0}>0\quad\text{for each}\  t\in[-t_0,t_0].
\end{equation*}
where constant $c_0$ is independent of $t$, $U(t)$ is the solution to equation (\ref{2.5}) with $\pL_1=\pL_1(t)$ and $U(t)$ is orthogonal to $\psi_0$ in $L_2(\square)$. Then all the above results remain true since their proofs remain unchanged.
\end{remark}

The structure of the paper is as follows: The next section presents various specific examples which are covered
by our general model. They were, in fact, the motivation and origin for the choice of the abstract model.
Thereafter follows Section \ref{s:lower.bound} with the proof of the variational lower bound on the ground state energy on the finite segment
and Section \ref{s:CTe} with the proof of the abstract Combes-Thomas estimate. In Section \ref{s:probabilistic.estimates} the proofs of the
probabilistic estimates are provided.


\section{Examples covered by the general model}
\label{Examples}

In this section we provide several examples of perturbations covered by our results. Namely, we discuss particular cases of operators $\pL(t)$ satisfying assumptions (\ref{as1}), (\ref{as2}). In what follows, we check only this assumptions since they suffice
 to establish all results presented in Section \ref{s:results}.

\subsection{Linear perturbations with positive coupling constants}
Condition (A1) imposes a quite strict condition on the linear part of the perturbation $\pL$.
However, if we restrict our considerations to non-negative coupling constants $\omega_k$, then much more general linear perturbation are allowed.
To see this, we consider the situation $\pL_1=0$, $\pL_3=0$. Then (A1) is trivially satisfied, (A2) requires $c_0=(\pL_2\psi_0,\psi_0)>0$,
and $\pL(t)=t^2 \pL_2$.  Thus $\pL(\e \omega_k)=\e^2 \omega_k^2 \pL_2$, hence we have non-negative coupling constants $\omega_k^2$ as prefactors. Note that any random variable
$\eta\colon \Omega \to [0,1]$ can be written as $\eta=\omega_0^2$ for some random variable $\omega_0\colon \Omega \to [-1,1]$, so in the case of non-negative random coupling constants the power two is no restriction. In this situation Theorem \ref{th3.1} gives:

\begin{theorem}
Let $\delta\in (0,1)$, $\eta \in \times_{k\in\G} [0,1]$, $\a\in\G$, $N\in\NN$ and
\begin{equation*}
\Op^\delta_{\a,N}(\omega):=-\D+V_0+\delta\sum\limits_{k\in\G_{\a,N}}  \eta_k\S(k) \pL_2 \S(-k)
\end{equation*}
with domain as in (\ref{operator.domain}). For sufficiently small $\delta$ this is a selfadjoint operator.

Then there exist positive constants $c_{1}$, $c_{2}$, $N_{1}$ such that for
\begin{equation}
N>N_{1},\quad \text{ and } \quad 0<\delta<\frac{c_{1}^2}{N^8}
\end{equation}
the estimate
\begin{equation*}
\l(\Op^\delta_{\a,N}(\eta))-\L_0\geqslant \frac{c_{2} \delta}{N^n} \sum\limits_{k\in\G_{\a,N}} \eta_k
\end{equation*}
holds true.
\end{theorem}
Here $\l(\Op^\delta_{\a,N}(\eta))$ denotes the lowest eigenvalue of
$\Op^\delta_{\a,N}(\eta)$. The theorem covers the case where the random variables are non-negative, the perturbation is linear and in an average sense positive.

In the present situation Corollary \ref{corollary:Lifschitz-tail-regime} takes the form of

\begin{corollary}\label{corollary:Lifschitz-tail-regime:linear}
Let $\alpha \in \Gamma$, $ \gamma \in \NN, \gamma\geqslant 17$ and $c_{1}$ be as in (\ref{eq:epsilon-N-relation}).
Choose $\delta >0$ and $N \in \NN$, $N\geqslant \max\{N_1^\gamma,K_1^\gamma,N_2\}$, such that $N= (\delta/c_{1}^2)^{-\gamma/8}$.
Let  $\b_1,\b_2$, $m_1, m_2$,  $B_1$, $B_2$, $c_{5}$ be as in Theorem  \ref{th2.2}.
Then
\begin{align*}
\PP_\eta\left(\forall \, \lambda \leqslant \Lambda_0 + \frac{1}{2} \left(\frac{\delta}{c_{1}^2}\right)^{\gamma/16}
 :\, \|\chi_{B_1} (\Op^\delta_{\a,N}(\eta)-\l)^{-1} \chi_{B_2} \|
 \leqslant 2  \left(\frac{\delta}{c_{1}^2}\right)^{-\gamma/16} \E^{-c_{5}\dist(B_1,B_2) \left(\frac{\delta}{c_{1}^2}\right)^{\gamma/16}}\right)
\\
\geqslant
1-   \left(\frac{\delta}{c_{1}^2}\right)^{-n(\gamma-1)/4} \E^{-c_{4}  \left(\frac{\delta}{c_{1}^2}\right)^{-\frac{n}{4}}}.
\end{align*}
\end{corollary}
Here $\PP_\eta$ denotes the distribution measure of the stochastic process $\eta$.

\subsection{Potential.} The canonical example is the perturbation by a potential:
\begin{equation*}
\pL(t)=t V_1 + t^2 V_2.
\end{equation*}
Here $V_1$, $V_2$ are measurable bounded real-valued functions
defined on $\square$ and $\pL_1$, $\pL_2$ are just operators of multiplication by $V_1$, $V_2$.

Assumption~(\ref{as1}) reads as
\begin{equation}\label{5.1}
\int\limits_{\square} V_1\psi_0^2\di x=0,
\end{equation}
while Assumption~(\ref{as2}) takes the form
\begin{equation}\label{5.2}
\int\limits_{\square} V_2\psi_0^2\di x > \int\limits_{\square} V_1 U\psi_0\di x.
\end{equation}
Here $U$ solves the boundary value problem described in Assumption~(\ref{as2}) with the right hand side $\pL_1\psi_0=V_1\psi_0$.
Since $\psi_0^2>0$ is positive everywhere, (\ref{5.1})
implies that either $V_1 =0$ almost everywhere, or $V_1$ changes sign.
It is clear that given $V_1$ obeying (\ref{5.1}), there is a wide class of potentials $V_2$ satisfying (\ref{5.2}).
\bigskip

\noindent
\emph{In all the following examples we assume that the boundary condition (\ref{2.8}) consists of Dirichlet boundary conditions.}

\subsection{Magnetic field.} The next example is a random magnetic field. The perturbed operator reads as
\begin{equation*}
\Op^\e(\omega)=(\iu\nabla+A^\e)^2,\quad A^\e=\e\sum\limits_{k\in\G_{\a,N}} \omega_k \S(k) A \S(-k).
\end{equation*}
Here $A=A(x)=(A_1(x),\ldots,A_{n+1}(x))$ is a real-valued magnetic field which is assumed to belong to $C^1(\overline{\square})$ and vanishing on the boundary of $\p\square'$:
\begin{equation}\label{5.3}
A(\cdot,x_{n+1})=0\quad\text{on}\quad \p\square'\quad\text{for each}\quad x_{n+1}\in(0,d).
\end{equation}
Since
\begin{equation*}
(\iu\nabla+A^\e)^2=-\D+2\iu A^\e\cdot\nabla+\iu \Div A^\e +|A^\e|^2,
\end{equation*}
operators $\pL_1$, $\pL_2$, $\pL_3$ are given by the identities
\begin{equation*}
\pL_1=2\iu A\cdot\nabla+\iu\Div A,\quad \pL_2=|A|^2,\quad \pL_3=0.
\end{equation*}
Let us check Assumption~(\ref{as1}).
We calculate
\begin{equation*}
(\pL_1\psi_0,\psi_0)_{L_2(\square)}
=  \iu \int\limits_{\square} \psi_0 \left(2 A_{n+1} \frac{\p\psi_0}{\p x_{n+1}} + \psi_0 \Div A\right)\di x
=\iu \int\limits_{\square} \left(A_{n+1} \frac{\p\psi_0^2}{\p x_{n+1}}+\psi_0^2\Div A\right)\di x.
\end{equation*}
Now we integrate by parts employing Dirichlet boundary conditions for $\psi_0$ and $A$:
\begin{align*}
 \iu \int\limits_{\square} \left(A_{n+1} \frac{\p\psi_0^2}{\p x_{n+1}}+\psi_0^2\Div A\right)\di x
=&\iu\int\limits_{\square}\psi_0^2 \sum\limits_{j=1}^{n} \frac{\p A_j}{\p x_j}\di x
\\
=&\iu \int\limits_{0}^{d}\di x_{n+1} \psi_0^2(x_{n+1}) \int\limits_{\square'} \sum\limits_{j=1}^{n} \frac{\p A_j}{\p x_j}(x',x_{n+1})\di x'=0
\end{align*}
since for each $\quad x_{n+1}\in(0,d)$
\begin{equation*}
\sum\limits_{j=1}^{n} \int_{\square'} \frac{\p A_j}{\p x_j}(x',x_{n+1})\di x'=
\sum\limits_{j=1}^{n} \int_0^1 \di x_1  \ldots \int_0^1 \di x_n \frac{\p A_j}{\p x_j}(x',x_{n+1})=
0 \quad 
\end{equation*}
which can be checked by integration by parts.

To check Assumption~(\ref{as2}), we first observe that $U=\iu \tU$, where $\tU$ is orthogonal to $\psi_0$ in $L_2(\square)$ and solves the equation
\begin{equation*}
(\Op_\square-\L_0)\tU=2A\cdot\nabla\psi_0+\psi_0\Div A.
\end{equation*}
Hence,
\begin{equation}\label{5.5}
\begin{aligned}
(U,\pL_1\psi_0)_{L_2(\square)}=&(\tU,2A\cdot\nabla\psi_0+\psi_0\Div A)_{L_2(\square)}
\\
=&\|\nabla \tU\|_{L_2(\square)}^2+(V_0\tU,\tU)_{L_2(\square)} -\L_0\|\tU\|_{L_2(\square)}^2,
\end{aligned}
\end{equation}
and
\begin{equation*}
c_{0}=\int\limits_{\square} \big(|A|^2\psi_0^2-2\tU A\cdot \nabla\psi_0-\psi_0\tU \Div A\big)\di x.
\end{equation*}

Let us prove that
\begin{equation}\label{5.7}
c_{0}=\int\limits_{\square} \Big|A+\nabla \frac{\tU}{\psi_0}\Big|^2\di x.
\end{equation}
We first observe that functions $\tU$ and $\psi_0$ satisfy the same boundary condition on $\p\square\cap\p\Pi$ and this is why function $\frac{\tU}{\psi_0}$ is well-defined and belongs at least to $H^1(\square)$.

To prove (\ref{5.7}), let us calculate the difference
\begin{equation*}
J:=\int\limits_{\square} \bigg(\Big|A+\nabla \frac{\tU}{\psi_0}\Big|^2- |A|^2\psi_0^2+2\tU A\cdot \nabla\psi_0+\psi_0\tU \Div A\bigg)\di x.
\end{equation*}
Since
\begin{equation*}
\psi_0\nabla \frac{\tU}{\psi_0}=\nabla \tU - \frac{\tU}{\psi_0} \nabla\psi_0,
\end{equation*}
we get:
\begin{align*}
J=\int\limits_{\square} \bigg(2\psi_0 A\cdot\nabla \tU + |\nabla \tU|^2 -2\frac{\tU}{\psi_0}\nabla\tU\cdot \nabla \psi_0 + \frac{\tU^2}{\psi_0^2}|\nabla\psi_0|^2 + \psi_0\tU\Div A \bigg)\di x.
\end{align*}
Integrating by parts, we obtain:
\begin{align*}
-2\int\limits_{\square} \frac{\tU}{\psi_0}\nabla\tU\cdot \nabla \psi_0 \di x=&-\int\limits_{\square} \nabla\tU^2\cdot \frac{1}{\psi_0} \nabla \psi_0\di x=\int\limits_{\square} \tU^2 \Div\frac{\nabla \psi_0}{\psi_0}\di x
\\
= &\int\limits_{\square} \bigg(-\L_0\tU^2-V_0\tU^2 - \frac{\tU^2}{\psi_0^2} |\nabla\psi_0|^2 \bigg)\di x.
\end{align*}
Hence, by (\ref{5.5}),
\begin{align*}
J=&\int\limits_{\square} \big(2\psi_0 A\cdot \Div \tU+ |\nabla\tU|^2-V_0\tU^2-\L_0\tU^2+\psi_0\tU\Div A\big)\di x
\\
=&2\int\limits_{\square} \big(\psi_0 A\cdot \Div \tU+\psi_0\tU\Div A+\tU A\cdot\nabla\psi_0\big)\di x=2\int\limits_{\square} \Div \psi_0\tU A\di x=0,
\end{align*}
where the latter identity has been obtained by integration by parts. Hence, identity (\ref{5.7}) holds true and therefore, Assumption~(\ref{as2}) is satisfied. It means that we can apply the results of the present paper to each weak random magnetic field provided (\ref{5.3}) is satisfied.

\subsection{Metric perturbation.} One more possible example is a random perturbation of metric. Here operator $\Op^\e(\omega)$ reads as
\begin{equation*}
\Op^\e(\omega)=-\D+V_0-\sum\limits_{k\in\G} \sum\limits_{i,j=1}^{n+1} \frac{\p\hphantom{x}}{\p x_i} \big(\e \omega_k a_{ij}(x'-k,x_{n+1})+\e^2 \omega_k^2 b_{ij}(x'-k,x_{n+1})\big) \frac{\p\hphantom{x}}{\p x_j},
\end{equation*}
where
$a_{ij}\colon \square \to \CC$, $b_{ij}\colon \square \to \CC$, $(i,j=1, \ldots,n+1)$ are complex-valued functions belonging to $C^1(\overline{\square})$, vanishing on $\p\square'\times[0,d]$, and satisfying the symmetry conditions
\begin{equation*}
a_{ij}=\overline{a}_{ji},\quad b_{ij}=\overline{b}_{ji}\quad \text{in}\quad \overline{\square}.
\end{equation*}
The operators $\pL_1$, $\pL_2$, $\pL_3$ are given by the identities \begin{equation*}
\pL_1=-\sum\limits_{i,j=1}^{n+1} \frac{\p\hphantom{x}}{\p x_i} a_{ij}(x)\frac{\p\hphantom{x}}{\p x_j},\quad \pL_2=-\sum\limits_{i,j=1}^{n+1} \frac{\p\hphantom{x}}{\p x_i} b_{ij}(x)\frac{\p\hphantom{x}}{\p x_j},\quad \pL_3=0.
\end{equation*}
Integrating by parts, we rewrite Assumption~(\ref{as1}) as
\begin{equation}\label{5.10}
0=\int\limits_{\square} \psi_0\pL_1\psi_0\di x=\int\limits_{\square} \sum\limits_{i,j=1}^{n+1} a_{ij} \frac{\p\psi_0}{\p x_j} \frac{\p\psi_0}{\p x_i}\di x= \int\limits_{\square} a_{n+1\,n+1} \Big(\frac{d\psi_0}{dx_{n+1}}\Big)^2\di x.
\end{equation}
This identity holds true for a wide class of functions $a_{n+1\,n+1}$. The simplest example is $a_{n+1\,n+1}=0$.
We stress that (\ref{5.10}) makes no restrictions for other coefficients $a_{ij}$, $(i,j)\not=(n+1,n+1)$.

Assumption~(\ref{as2}) here looks as
\begin{equation}\label{5.11}
\begin{aligned}
&\int\limits_{\square} \left(b_{n+1\,n+1}\Big(\frac{d\psi_0}{dx_{n+1}}\Big)^2
-\sum\limits_{i,j=1}^{n+1} a_{ij} \frac{\p\psi_0}{\p x_j}\frac{\p U}{\p x_i}\right)\di x
\\
&=\int\limits_{\square} \left(b_{n+1\,n+1}\Big(\frac{d\psi_0}{dx_{n+1}}\Big)^2
-\sum\limits_{i=1}^{n+1} a_{i\,n+1} \frac{\p\psi_0}{\p x_{n+1}}\frac{\p U}{\p x_i}\right)\di x>0,
\end{aligned}
\end{equation}
where $U$ is orthogonal to $\psi_0$ in $L_2(\square)$ and solves the equation
\begin{equation*}
(\Op_\square-\L_0)U=-\sum\limits_{i=1}^{n+1} \frac{\p a_{i\,n+1}}{\p x_i} \frac{\p\psi_0}{\p x_{n+1}}.
\end{equation*}
Inequality (\ref{5.11}) is satisfied by a wide class of functions $b_{ij}$, $a_{ij}$. Here the simplest example is
\begin{equation*}
b_{n+1\,n+1}>0,\quad a_{i\,n+1}=a_{n+1\,i}=0
\end{equation*}
and other coefficients are arbitrary.
(Here we also have to assume that the operator $\Op_\square$ does not have a constant function as the ground state,
as it happens when $V_0=0$ and the boundary conditions in (\ref{2.3}) are of Neumann type.)
Then the right hand side of the above equation for $U$ vanishes and the left hand side in (\ref{5.11}) reduces to
$
\int\limits_{\square} \b_{n+1\,n+1} \big(\frac{d\psi_0}{dx_{n+1}}\big)^2 \di x$, which is a strictly positive integral.

\subsection{Integral operator.} The operators $\pL_i$ need not necessarily be differential expressions, as above, since we make very weak assumptions in their definition. An example of a non-differential operator is an integral operator:
\begin{equation*}
(\pL_i u)(x)=\int\limits_{\square} K_i(x,y)u(y)\di y,\quad i=1,2,\quad \pL_3=0,
\end{equation*}
where $K_i$, $i=1,2$, are measurable  functions defined on $\square\times\square$ and obeying the symmetry condition $K_i(x,y)=\overline{K_i(y,x)}$, $(x,y)\in\square\times\square$, $i=1,2$.

Assumption (\ref{as1}) is equivalent to vanishing of certain mean for $K_1$:
\begin{equation*}
\int\limits_{\square\times\square} K_1(x,y)\psi_0(x)\psi_0(y)\di x\di y=0.
\end{equation*}
If we suppose
$K_1(x,y)=K(x)\overline{K(y)}$,
Assumption (\ref{as1}) becomes equivalent to
\begin{equation*}
\int\limits_{\square} K(x)\psi_0(x)\di x=0
\end{equation*}
and it implies that $\pL_1\psi_0=0$, $U=0$. Then Assumption (\ref{as2}) holds true provided
\begin{equation*}
\int\limits_{\square\times\square} K_2(x,y)\psi_0(x)\psi_0(y)\di x\di y>0,
\end{equation*}
and this inequality is satisfied by a wide class of kernels $K_2$. For instance, the latter inequality holds true provided kernel $K_2$ is non-negative
and does not vanish identically.

\subsection{Boundary deformation.} Our next example is devoted to a geometric perturbation. Let $y=(y',y_{n+1})$, $y'=(y_1,\ldots,y_n)$ be Cartesian coordinates in $\RR^n$ and $\RR^{n+1}$, and $g=g(y')$ be a non-zero real-valued  function defined on $\RR^n$ and belonging to $C^2(\RR^n)$. We suppose that the support of $g$ is located inside $\square'$, i.e. $g$ vanishes in a vicinity of $\p\square'$ and outside $\square$. We introduce the function
\begin{equation*}
g_\omega^\e(y')=\sum\limits_{k\in\G} \e \omega_k g(x'-k,x_{n+1}).
\end{equation*}
It is equal to $\e \omega_k g(y'-k,y_{n+1})$ on $k+\square'$.
Employing this function, we define a weak random perturbation of the layer $\Pi$:
\begin{equation*}
\Pi^\e:=\{y:\, y'\in \RR^n,\, g_\omega^\e(y')<y_{n+1}<g_\omega^\e(y')+d\}.
\end{equation*}
The boundary of $\Pi^\e$ can be regarded as a weak random wiggling of $\p\Pi$.

In $\Pi^\e$ we consider the Dirichlet Laplacian, which we denote by $\tOp^\e(\omega)$.
The operator $\tOp^\e(\omega)$ does not satisfy our assumptions since it is defined on a domain $\Pi^\e$ depending on a small parameter. But it is possible to transform this operator to make it fit our model.
Namely, one can verify by direct calculation
that changing variables $x'=y'$, $x_{n+1}=y_{n+1}-g_\omega^\e(y')$, we keep the spectrum of $\tOp^\e(\omega)$ unchanged and we arrive at the operator
\begin{equation}\label{5.16}
\Op^\e(\omega)=-\D-\Div P_\omega^\e \nabla\quad \text{in}\quad \Pi,
\end{equation}
where $P_\omega^\e$ is $(n+1)\times (n+1)$ matrix defined by
\begin{equation*}
P_\omega^\e=
\begin{pmatrix}
0 & \nabla' -g_\omega^\e
\\
-(\nabla' g_\omega^\e)^t  & |\nabla' g_\omega^\e|^2
\end{pmatrix},\quad \nabla' g_\omega^\e=
\begin{pmatrix}
\frac{\p g_\omega^\e}{\p x_1}
\\
\vdots
\\
\frac{\p g_\omega^\e}{\p x_{n}}
\end{pmatrix}.
\end{equation*}
The operator (\ref{5.16}) corresponds to  (\ref{2.2}) with
\begin{equation*}
\pL_1=\sum\limits_{j=1}^{n} \frac{\p\hphantom{x}}{\p x_{n+1}} \frac{\p g}{\p x_j} \frac{\p\hphantom{x}}{\p x_j} + \frac{\p\hphantom{x}}{\p x_j}\frac{\p g}{\p x_j}\frac{\p\hphantom{x}}{\p x_{n+1}},\quad \pL_2=-|\nabla' g|^2\frac{\p^2\hphantom{x}}{\p x_{n+1}^2},\quad g=g(x'),\quad \pL_3=0.
\end{equation*}

We proceed to checking Assumptions~(\ref{as1}),~(\ref{as2}). Integrating by parts and taking into consideration that $g$ vanishes in the vicinity of $\p\square'$, we get
\begin{equation*}
\int\limits_{\square} \psi_0\pL_1\psi_0\di x=\int\limits_{\square} \sum\limits_{j=1}^{n} \frac{\p\hphantom{x}}{\p x_j}\frac{\p g}{\p x_j}{d\psi_0}{d x_{n+1}}\di x=0
\end{equation*}
and Assumption~(\ref{as1}) is satisfied.

To check Assumption~(\ref{as2}), we first observe that
\begin{equation*}
\pL_1\psi_0=\frac{d\psi_0}{d x_{n+1}} \D_{x'} g,
\end{equation*}
and, integrating by parts,
\begin{equation}\label{5.18}
\int\limits_{\square} \psi_0 \pL_2\psi_0\di x=-\int\limits_{\square} |\nabla' g|^2 \frac{d^2\psi_0}{d x_{n+1}^2}\psi_0\di x=\L_0\int\limits_{\square'} |\nabla' g|^2\di x'.
\end{equation}

The equation for $U$ reads as
\begin{equation}\label{5.13}
(-\D-\L_0)U=\frac{d\psi_0}{dx_{n+1}}\D_{x'} g
\end{equation}
and thus, integrating by parts,
\begin{align*}
\int\limits_{\square} \frac{d\psi_0}{dx_{n+1}}U\D_{x'} g\di x
&=\int\limits_{\square} \frac{d\psi_0}{dx_{n+1}}g\D_{x'} U\di x
\\
&=\int\limits_{\square} \frac{d\psi_0}{dx_{n+1}}g
\left(\left(-\frac{d^2\hphantom{x}}{dx_{n+1}^2}-\L_0\right)U-
\frac{d\psi_0}{dx_{n+1}}\D_{x'} g\right)\di x
\\
&= - \int\limits_{\square} \Big(\frac{d\psi_0}{dx_{n+1}}\Big)^2g\D_{x'} g\di x
-\int\limits_{\square} \frac{d\psi_0}{dx_{n+1}}g
\left(\frac{d^2\hphantom{x}}{dx_{n+1}^2}+\L_0\right)U\di x
\\
&=  \L_0\int\limits_{\square'}|\nabla' g|^2\di x' -\int\limits_{\square} \frac{d\psi_0}{dx_{n+1}}g
\left(\frac{d^2\hphantom{x}}{dx_{n+1}^2}+\L_0\right)U\di x.
\end{align*}
Together with (\ref{5.18}) it implies the formula for $c_{0}$:
\begin{equation}\label{5.19}
c_{0}=\int\limits_{\square} \frac{d\psi_0}{dx_{n+1}}g
\left(\frac{d^2\hphantom{x}}{dx_{n+1}^2}+\L_0\right)U\di x.
\end{equation}

To check the sign of $c_{0}$, we solve equation (\ref{5.13}) by separation of variables. Namely, since $$\psi_0=\sqrt{\frac{2}{d}}\sin \frac{\pi}{d} x_{n+1},$$
we can write the Fourier series
\begin{equation*}
\frac{d\psi_0}{dx_{n+1}}=\sum\limits_{m=1}^{\infty} a_m \psi_m,\quad \psi_m(x_{n+1}):=\sqrt{\frac{2}{d}} \sin \frac{\pi m}{d} x_{n+1},\quad a_m:=\int\limits_{0}^{d} \frac{d\psi_0}{dx_{n+1}} \psi_m \di x_{n+1}.
\end{equation*}
Then we represent $U$ as
\begin{equation*}
U(x)=\sum\limits_{m=1}^{\infty} a_m U_m(x') \psi_m(x_{n+1}),
\end{equation*}
and obtain that $U_m$ should solve the equation
\begin{equation}\label{5.20}
\left(-\D_{x'}+\frac{\pi^2}{d^2}(m^2-1)\right)U_m=\D_{x'}g\quad\text{in}\quad \square'
\end{equation}
subject to Neumann condition on $\p\square'$. We substitute the above Fourier series for $\frac{d\psi_0}{d x_{n+1}}$ and $U$ into (\ref{5.19}) to obtain
\begin{equation}\label{5.21}
c_{0}=-\frac{\pi^2}{d^2} \sum\limits_{m=1}^{\infty} a_m^2 (m^2-1) (g,U_m)_{L_2(\square')}.
\end{equation}

We represent function $U_m$ as
\begin{equation}\label{5.24}
 U_m=-g+W_m,
\end{equation}
and in view of (\ref{5.20}), $W_m$ solves the equation
\begin{equation*}
\left(-\D_{x'}+\frac{\pi^2}{d^2}(m^2-1)\right)W_m=
\frac{\pi^2}{d^2}(m^2-1)g\quad\text{in}\quad \square'
\end{equation*}
subject to Neumann condition on $\p\square'$. It yields
\begin{equation}\label{5.23}
\begin{aligned}
\|\nabla' W_m\|_{L_2(\square')}^2 +\frac{\pi^2}{d^2}(m^2-1) \|W_m\|_{L_2(\square')}^2 = &
\frac{\pi^2}{d^2}(m^2-1)(g,W_m)_{L_2(\square')}
\\
\leqslant & \frac{\pi^2}{d^2}(m^2-1) \|g\|_{L_2(\square')}\|W_m\|_{L_2(\square')}.
\end{aligned}
\end{equation}
Hence,
\begin{equation*}
\|W_m\|_{L_2(\square')}<\|g\|_{L_2(\square')}.
\end{equation*}
Here we have a strict inequality, since in the case of identity, it follows from (\ref{5.23}) that $\nabla' W_m=0$ and $W_m=\mathrm{const}$ that contradicts equation for $W_m$. It follows from the obtained inequality and (\ref{5.24}) that
$(g,U_m)_{L_2(\square')}<0$ for each $m\geqslant 1$. Therefore, each term in the series in the right hand side of (\ref{5.21})
is negative and $c_{0}>0$. Thus, our operator satisfies Assumptions~(\ref{as1}),~(\ref{as2}) and we can apply the results of this paper to a weak random wiggling of the boundary.

\subsection{Random operators in multi-dimensional spaces.}
\label{ss:whole.space}
Now we show that our setting covers not only operators defined in a finite-width layer in $\RR^{n+1}$, but operators defined on the whole Euclidean space as well.
Recall $\square':=\{x': x'=\sum\limits_{i=1}^{n} a_i e_i,\; a_i\in(0,1)\}\subset \RR^n$.
Let
\begin{equation*}
\pL'(t):=t \pL'_1 + t^2 \pL'_2 + t^3 \pL'_3(t),
\end{equation*}
where $\pL'_i: {H^2}(\square')\to L_2(\square')$ are bounded symmetric linear operators and $\pL'_3(t)$ is bounded uniformly in $t\in[-t_0,t_0]$.
In $n$-dimensional Euclidean space we consider the operator
\begin{equation*}
{\Op'}^\e(\omega):=-\D_{x'}+ \sum\limits_{k\in\G} \S'(k) \pL'(\e \omega_k) \S'(-k),
\end{equation*}
where $\D_{x'}$ is the Laplacian in $\RR^n$ and $\S'(k)$ is a shift operator: $(\S'(k)u)(x')=u(x'+k)$.
Then ${\Op'}^\e$ is a random self-adjoint operator in $L_2(\RR^n)$ with a similar structure as  $\Op^\e$.
The only difference is that it acts on functions in $\RR^n$.

Based on the  $\pL'_i$ we define operators $\pL_i: {H^2}(\square)\to L_2(\square)$:
\begin{equation*}
(\pL_i u)(x',x_{n+1})=\pL'_i u(\cdot,x_{n+1}).
\end{equation*}
Thus $\pL'_i$ acts on $x'\mapsto u(x', x_{n+1})$, while $x_{n+1}$ is regarded as a parameter.
The result is a function depending on $x'$ and $x_{n+1}$: it is precisely $\pL_i u$.

Now that we have $\pL_1$, $\pL_2$, and $\pL_3$ at our disposal, the operator $\Op^\e$ is defined as in (\ref{2.2}).
We choose Neumann condition on $\p\Pi$, $V_0=0$,  and $d=\pi$.
The spectrum of $\Op^\e$ can be found by separating the variables $x'$ and $x_{n+1}$. Namely,
\begin{equation}\label{5.24a}
\spec(\Op^\e)=\bigcup\limits_{m=0}^{\infty} \spec(\Op^\e+m^2)
\end{equation}
since we can represent each function $u$ in the domain of $\Op^\e$ by its Fourier series:
\begin{equation*}
u(x)=\sum\limits_{m=0}^{\infty} u_m(x')\cos m x_{n+1}.
\end{equation*}
The eigenvalue $\L_0$ in our case vanishes: $\L_0=0$. Assumptions~(\ref{as1}),~(\ref{as2}) take on the form:

\begin{enumerate}\def\theenumi{A\arabic{enumi}'}

\item\label{as1p} The identity
$\int\limits_{\square'} \pL'_1 \mathds{1} dx'=0$
holds true, where $\mathds{1}(x')=1$ in $\square'$.

\item\label{as2p} Let $U'$ be the unique solution to the two equations
\begin{equation*}
\Op_{\square'}U'=\pL_1\mathds{1}, \quad
\int\limits_{\square'} U' dx'=0.
\end{equation*}
Here $\Op_{\square'}$ is the negative Neumann Laplacian on $\square'$. We assume that
\begin{equation*}
c'_{0}:=\int\limits_{\square'}\pL_2\mathds{1} dx -(U',\pL'_1\mathds{1})_{L_2(\square')}>0.
\end{equation*}
\end{enumerate}
Once these assumptions are satisfied, by (\ref{5.24a}) and Theorems~\ref{th2.1},~\ref{th2.2} we obtain immediately the analogues of these theorems for ${\Op'}^\e$.


\begin{theorem}[The result described in the introduction]
\label{th: whole.space}
There exist positive constants $c_{1}'$, $c_{2}'$, $N_{1}'$ such that for
\begin{equation}
N>N_{1}', \text{ and } 0<\e<\frac{c_{1}'}{N^4}
\end{equation}
the estimate
\begin{equation}
\l_{\a,N}^\e(\omega)-\L_0\geqslant \frac{c_{2}' \e^2 }{N^n} \sum\limits_{k\in\G_{\a,N}} \omega_k^2
\end{equation}
holds true.
\end{theorem}

\begin{theorem}\label{th5.1}
Given $\g\in\NN, \g \geqslant 17$,  there exist constants $c_{3}'$, $c_{4}'$, $N_1'$ such that for $N>N_1'$ the interval
\begin{equation*}
I_N:=\left[\frac{c_{3}' }{\EE(|\omega_k|)N^{\frac{1}{4}}}, \frac{c_{1}'}{N^{\frac{4}{\g'}}}\right]
\end{equation*}
is non-empty. For $N>N_1'$ and $\e\in I_N$, the estimate
\begin{equation*}
\PP\left(\omega\in\Om:\, \l_{\a,N}^\e\leqslant N^{-\frac{1}{2}}\right)\leqslant N^{n\left(1-\frac{1}{\g}\right)}\E^{-c_{4}'N^{\frac{n}{\g}}}
\end{equation*}
holds true, where $c_{4}'$ depends on $\mu$ only.
\end{theorem}

In this theorem $\l_{\a,N}^\e$ is the lowest eigenvalue of the operator ${\Op'}^\e_{\a,N}$.
The latter is the restriction of ${\Op'}^\e$ to
\begin{equation*}
\Pi'_{\a,N}:=\big\{x'\in \RR^n: \, x'=\a+\sum\limits_{i=1}^{n} a_i e_i,\, a_i\in(0,N)\big\}
\end{equation*}
with Neumann boundary conditions.

\begin{theorem}\label{th2.2bis}
Assume the hypothesis of Theorem~\ref{th5.1}, let $\e\in I_N$ and fix $\b_1,\b_2\in\G_{\a,N}$, $m_1, m_2>0$ such that $B_1:=\Pi'_{\b_1,m_1}\subset\Pi'_{\a,N}$, $B_2:=\Pi'_{\b_2,m_2}\subset\Pi'_{\a,N}$. Then there exists a constant $c_{5}'$ independent of $\e$, $\a$, $N$, $\b_1$, $\b_2$, $m_1$, $m_2$ such that  for $N\geqslant N_1'$
\begin{align*}
\PP\left(\omega\in\Om:\, \|\chi_{B_1} ({\Op'}^\e_{\a,N}-\l)^{-1} \chi_{B_2} \| \leqslant 2\sqrt{N} \E^{-\frac{c_{5}'\dist(B_1,B_2)}{\sqrt{N}}}
\right)\geqslant 1-N^{n\left(1-\frac{1}{\g}\right)} \E^{-c_{4}' N^{\frac{n}{\g}}},
\end{align*}
where $\|\cdot\|$ stands for the norm of an operator in $L_2(\Pi'_{\a,N})$.
\end{theorem}



\section{Deterministic lower bound}
\label{s:lower.bound}

The essential milestone in proving our main result is a deterministic variational estimate provided in Theorem~\ref{th:deterministic.lower.bound}. For the reader's convenience we formulate it here once again.

\begin{theorem}[Theorem \ref{th:deterministic.lower.bound} above]
\label{th3.1}
There exist positive constants $c_{1}$, $c_{2}$, $N_{1}$ such that for
\begin{equation}\label{3.1}
N>N_{1}\quad \text{ and } \quad 0<\e<\frac{c_{1}}{N^4}
\end{equation}
the estimate
\begin{equation}\label{3.2}
\l_{\a,N}^\e(\omega)-\L_0\geqslant \frac{c_{2} \e^2}{N^n} \sum\limits_{k\in\G_{\a,N}} \omega_k^2
\end{equation}
holds true.
\end{theorem}

The rest of this section is devoted to the proof of the above theorem.
Throughout the proof by $C$ we denote  various constants independent of $\e$ and $N$.

\subsection{Setup for analytic perturbation theory}

We begin with considering operators $\Op^0_{\a,N}(0)$, i.e., the Schr\"odinger operator $-\D+V_0$ in $\Pi_{\a,N}$ subject to boundary condition (\ref{2.8}) on $\g_{\a,N}$ and to Neumann condition on $\p\Pi_{\a,N}\setminus\overline{\g_{\a,N}}$. The lowest eigenvalue of operator is $\l_{\a,N}^0=\L_0$ and the associated eigenfunction normalized in $L_2(\Pi_{\a,N})$ is $N^{-\frac{n}{2}}\psi_0$. Provided $N>N_{1}$ and $N_{1}$ is large enough, the second eigenvalue is $\L_0+\frac{\k}{N^2}$, where $\k>0$ is the second eigenvalue of the negative Neumann Laplacian on $\square'$.  Then there exists $C >0$ such that the ball
$$
B:=\big\{\l\in\mathds{C}:\, |\l-\L_0|\leqslant  CN^{-2}\big\}
$$
contains no eigenvalues of $\Op_{\a,N}^0(0)$ except $\L_0$ and the distance from $B$ to all the eigenvalues of $\Op_{\a,N}^0(0)$ except $\L_0$ is estimated from below by $C N^{-2}$.

For $\l\in B\setminus \{ \Lambda_0\}$ the resolvent $(\Op_{\a,N}^0-\l)^{-1}$ is represented as
\begin{equation*}
(\Op_{\a,N}^0-\l)^{-1}= \frac{1}{N^n} \frac{(\,\cdot\,,\psi_0)_{L_2(\Pi_{\a,N})}}{\L_0-\l}\psi_0 + \Rs_{\a,N}(\l),
\end{equation*}
where $\Rs_{\a,N}$ is the reduced resolvent. It is an operator from $L_2(\Pi_{\a,N})$ into ${H^2}(\Pi_{\a,N})$. Its range is orthogonal to $\psi_0$ in $L_2(\Pi_{\a,N})$. Moreover, by analogy with \cite[Lm. 5.2]{BorisovV-11} one can prove easily the following lemma.

\begin{lemma}\label{lm3.1}
For $\l\in B$  and $f\in L_2(\Pi_{\a,N})$ the estimate
\begin{equation*}
\|\Rs_{\a,N}f\|_{{H^2}(\Pi_{\a,N})}\leqslant CN^2\|f\|_{L_2(\Pi_{\a,N})},
\end{equation*}
where $C$ is a constant independent of $\l$, $N$, $f$.
\end{lemma}

At the next step we describe the minimum of $\l_{\a,N}^\e$ w.r.t. $\omega_k$.

\begin{lemma}\label{lm3.2}
\label{l:lifting}
(a)
For each fixed (sufficiently small) value of $\e\geqslant 0$
the minimum of $\l_{\a,N}^\e$ as a function of the variables  $\omega_k$, $k\in \G_{\a,N}$ is achieved for $\omega_k=0$, $k\in\G_{\a,N}$.

\noindent(b)
Indeed, there exists $\rho\in(0,\infty)$ independent of $\epsilon, \alpha, N$, and the configuration  $\Omega=(\omega_\alpha)_{k\in\G}$ such that
\begin{equation}
\forall \, \a \in \G, N\in \NN, \omega \in \Omega, \e \geqslant 0: \quad
\l_{\a,N}^\e(\omega) \geqslant  \L_0+\e^2   \min\limits_{k\in\G_{a,N}}\left ( c_{0}\omega_k^2 - \rho |\e \omega_k^3|\right)
\end{equation}
\noindent(c)
Consider the particular configuration $\tilde \omega \in \Omega$ with $\tilde \omega_k=1$ for all $k \in \G$. Then
\begin{equation}\label{eq:lifting.finite.periodic}
\forall \, \a \in \G, N\in \NN, \e \geqslant 0: \quad
\l_{\a,N}^\e(\tilde \omega) \geqslant  \L_0+c_{0}\e^2- \rho \e^3
\end{equation}
and
\begin{equation}\label{eq:lifting.infinite.periodic}
\forall \, \e \geqslant 0: \quad
\inf \sigma(\Op^\e(\tilde \omega)) \geqslant  \L_0+c_{0}\e^2- \rho \e^3
\end{equation}
where $\rho$ is the same constant as in (b).
\end{lemma}

\begin{proof}
We begin with the case $N=1$. Then $\l_{\a,1}^\e$ is the lowest eigenvalue of operator $\Op_{\a,1}^\e$.
This operator is considered in cell $\square_\a$ and it given by
\begin{equation*}
\Op_{\a,1}^\e=-\D+V_0+\S(-\a) \pL(\e \omega_\a) \S(\a).
\end{equation*}
By means of regular perturbation theory we can write the first terms of the asymptotics for $\l_{\a,1}^\e$:
\begin{align*}
\l_{\a,1}^\e=&\L_0 + (\pL(\e \omega_\a)\psi_0,\psi_0)_{L_2(\square)} + \big( \pL(\e \omega_\a) \Rs_{\a,1}(\L_0) \psi_0, \psi_0\big)_{L_2(\square)} + \Odr(\e^3 \omega_\a^3)
\\
=&\L_0 + \e \omega_\a (\pL_1\psi_0,\psi_0)_{L_2(\square)} + (\e \omega_\a)^2\big( \pL_2 \psi_0 - \pL_1  \Rs_{\a,1}(\L_0) \pL_1  \psi_0, \psi_0\big)_{L_2(\square)} + \Odr(\e^3 \omega_\a^3).
\end{align*}
We apply assumptions (\ref{as1}), (\ref{as2}) to simplify the above expansion.
By (\ref{as1}) the next-to-leading term vanishes and it is easy to infer form (\ref{as2}) that
\begin{equation*}
\big( \pL_2 \psi_0 - \pL_1  \Rs_{\a,1}(\L_0) \pL_1  \psi_0, \psi_0\big)_{L_2(\square)}=c_{0}>0.
\end{equation*}
Hence, there exists a constant $\rho$ independent of $\epsilon, \alpha, \omega_\alpha$ such that
\begin{equation*}
\l_{\a,1}^\e\geqslant \L_0+\e^2 \omega_\a^2 c_{0} - \rho |\e^3 \omega_\a^3|.
\end{equation*}
This identity implies that $\l_{\a,1}^\e$ achieves its minimum $\L_0$ as a function of  $\omega_\a$ for  $\omega_\a=0$.

We proceed to studying $\l_{\a,N}^\e$. In domain $\Pi_{\a,N}$ we introduce additional Neumann conditions on lateral boundaries $\p\square_k\setminus\p\Pi$ of $\square_k$ for each $k\in\G_{\a,N}$. By the minimax principle it gives the lower bound for $\l_{\a,N}^\e$:
\begin{equation}\label{3.5}
\l_{\a,N}^\e\geqslant \min\limits_{k\in\G_{a,N}} \l^\e_{k,1} \geqslant \L_0 +\e^2
\min\limits_{k\in\G_{a,N}}( c_{0}\omega_k^2 - \rho |\e \omega_k^3|)
\end{equation}
At the same time,  it is straightforward to check that as $\omega_k=0$, $k\in\G_{\a,N}$, the lowest eigenvalue of $\Op_{\a,N}^\e$ is $\L_0$ and the associated eigenfunction is $\psi_0$. This  completes the proof of (a) and (b), and
(\ref{eq:lifting.finite.periodic})
is a special case of (b). For the second bound
(\ref{eq:lifting.infinite.periodic})
in (c) we note that, as above, the introduction of additional Neumann boundary conditions yields
\[
\inf \sigma(\Op^\e(\tilde W)))
\geqslant \inf\limits_{k\in\G} \l_{\a,1}^\e(\tilde \omega)
\geqslant \L_0 +\e^2 \min\limits_{k\in\G_{a,N}}\left( c_{0} - \rho \e \right).
\]
\end{proof}

Let us show that $\l_{\a,N}^\e$ belongs to $B$. In accordance with (\ref{3.5}), $\l_{\a,N}^\e\geqslant \L_0$. By the minimax principle we obtain the upper estimate:
\begin{align*}
\l_{\a,N}^\e\leqslant & \frac{\|\nabla\psi_0\|_{L_2(\Pi_{\a,N})}^2 + (V_0\psi_0,\psi_0)_{L_2(\square)} +(\pL^\e\psi_0,\psi_0)_{L_2(\Pi_{\a,N})}}{\|\psi_0\|_{L_2(\Pi_{\a,N})}^2}
\\
=&\L_0+\frac{\sum\limits_{k\in\G_{\a,N}}\e \omega_k (\pL_1\psi_0,\psi_0)_{L_2(\square)}+\e^2 \omega_k^2 \big((\pL_2+\e \omega_k \pL_3(\e \omega_k))\psi_0,\psi_0\big)_{L_2(\square)}}{\|\psi_0\|_{L_2(\Pi_{\a,N})}^2}.
\end{align*}
By assumption (\ref{as1}), the sum $\sum\limits_{k\in\G_{\a,N}}\e \omega_k (\pL_1\psi_0,\psi_0)_{L_2(\square)}$ vanishes and we can continue estimating as follows,
\begin{equation*}
\l_{\a,N}^\e\leqslant  \L_0+\frac{\sum\limits_{k\in\G_{\a,N}}\e^2 \omega_k^2 \big((\pL_2+\e \omega_k \pL_3(\e \omega_k))\psi_0,\psi_0\big)_{L_2(\square)}}{\|\psi_0\|_{L_2(\Pi_{\a,N})}^2}
\leqslant  \L_0+C\e^2\leqslant \L_0+\frac{C}{N^4}
\end{equation*}
Thus, as $N>N_{1}$ and $N_{1}$ is great enough, $\l_{\a,N}^\e$ belongs to $B$.

\subsection{Non-self-adjoint Birman-Schwinger principle}
To obtain the desired deterministic estimate, we apply the non-self-adjoint modification of Birman-Schwinger principle proposed in \cite{Gadylshin-02}, in the same way as it was done in
\cite{Borisov-06}, \cite[Sect. 5]{BorisovV-11}. It leads us to the equation for $\l_{\a,N}^\e$:
\begin{equation}\label{3.6}
\l_{\a,N}^\e-\L_0=\frac{1}{N^n} \big((\I+\pL_{\a,N}^\e \Rs_{\a,N}(\l_{\a,N}^\e))^{-1}\pL_{\a,N}^\e\psi_0,\psi_0\big)_{L_2(\Pi_{\a,N})},
\end{equation}
where $\I$ denotes the identity mapping.
As $\l\in B$, by Lemma~\ref{lm3.1} and the boundedness of $\pL_i$ we have the estimate:
\begin{equation}\label{3.7}
\|\pL_{\a,N}^\e \Rs_{\a,N}(\l)\|\leqslant C\e N^2.
\end{equation}
Hereinafter $\|\cdot\|$ stands for the norm of operators in $L_2(\Pi_{\a,N})$. Thanks to the above estimate and (\ref{3.1}), we can conclude that for a properly chosen $N_{1}$
\begin{equation*}
\|\pL_{\a,N}^\e \Rs_{\a,N}(\l)\|\leqslant C<1.
\end{equation*}
Hence, operator $(\I+\pL_{\a,N}^\e \Rs_{\a,N}(\l))^{-1}$ is well-defined and can be estimated as
\begin{equation*}
\|\big(\I+\pL_{\a,N}^\e \Rs_{\a,N}(\l)\big)^{-1}\|\leqslant C
\end{equation*}
uniformly in $\e$, $N$, $\l$.

\subsection{Taylor expansion to third order}
Equation (\ref{3.6}) is the main tool in proving the desired estimate for $\l_{\a,N}^\e-\L_0$. We represent  $(\I+\pL_{\a,N}^\e \Rs_{\a,N}(\l))^{-1}$ as
\begin{equation*}
\big(\I+\pL_{\a,N}^\e \Rs_{\a,N}(\l)\big)^{-1}=\I-\pL_{\a,N}^\e \Rs_{\a,N}(\l) + \big(\pL_{\a,N}^\e \Rs_{\a,N}(\l)\big)^2 \big(\I+\pL_{\a,N}^\e \Rs_{\a,N}(\l)\big)^{-1}
\end{equation*}
and substitute this representation into (\ref{3.6}):
\begin{align*}
\l_{\a,N}^\e-\L_0=&\frac{1}{N^n} (\pL_{\a,N}^\e\psi_0,\psi_0)_{L_2(\Pi_{\a,N})} -
\frac{1}{N^n} \big(\pL_{\a,N}^\e \Rs_{\a,N}(\l_{\a,N}^\e)\pL_{\a,N}^\e\psi_0,\psi_0\big)_{L_2(\Pi_{\a,N})}
 \\
 &+ \frac{1}{N^n}
 \Big(\big(\pL_{\a,N}^\e \Rs_{\a,N}(\l_{\a,N}^\e)\big)^2(\I+\pL_{\a,N}^\e \Rs_{\a,N}(\l_{\a,N}^\e))^{-1}\pL_{\a,N}^\e\psi_0,\psi_0\Big)_{L_2(\Pi_{\a,N})}.
\end{align*}
We rewrite the obtained equation by employing the resolvent identity
\begin{equation*}
\Rs_{\a,N}(\l_{\e,N}^\a)-\Rs_{\a,N}(\L_0)=
\Big(\l_{\e,N}^\a-\L_0\Big)  \Rs_{\a,N}(\L_0) \Rs_{\a,N}(\l_{\e,N}^\a)
\end{equation*}
as follows:
\begin{equation}\label{3.10}
\begin{aligned}
\l_{\a,N}^\e-\L_0
=&\frac{1}{N^n} \Big((\pL_{\a,N}^\e\psi_0,\psi_0)_{L_2(\Pi_{\a,N})} -
(\pL_{\a,N}^\e \Rs_{\a,N}(\L_0)\pL_{\a,N}^\e\psi_0, \psi_0)_{L_2(\Pi_{\a,N})}
 \\
 &+
 \big((\pL_{\a,N}^\e \Rs_{\a,N}(\l_{\a,N}^\e)\big)^2(\I+\pL_{\a,N}^\e \Rs_{\a,N}(\l_{\a,N}^\e))^{-1}\pL_{\a,N}^\e\psi_0,\psi_0\big)_{L_2(\Pi_{\a,N})}\Big)
 \\
&
\cdot \left(1+\frac{1}{N^n}
\big(\pL_{\a,N}^\e
\Rs_{\a,N}(\L_0) \Rs_{\a,N}(\l_{\e,N}^\a)
\pL_{\a,N}^\e\psi_0,\psi_0\big)_{L_2(\Pi_{\a,N})}
\right)^{-1}.
\end{aligned}
\end{equation}

\subsection{Estimates on the individual terms}
To estimate the terms in the obtained identity we shall make use of the following auxiliary lemma. We recall that $\square_k:=\{x: x-(k,0)\in\square\}$.
\begin{lemma}\label{lm3.3}
For $i\in \{1,2,3\}$ and each $u\in {H^2}(\Pi_{\a,N})$ the inequalities
\begin{align*}
&\bigg|\Big(\sum\limits_{k\in\G_{\a,N}}\e \omega_k\S(-k)\pL_i  \S(k) u,\psi_0\Big)_{L_2(\Pi_{\a,N})}\bigg|\leqslant C\e\bigg(\sum\limits_{k\in\G_{\a,N}} |\omega_k|^2\bigg)^{\frac{1}{2}} \|u\|_{{H^2}(\Pi_{\a,N})},
\\
&\bigg\|\sum\limits_{k\in\G_{\a,N}}\e \omega_k \S(-k) \pL_i \S(k)u\bigg\|_{L_2(\Pi_{\a,N})}\leqslant C\e \bigg(\sum\limits_{k\in\G_{\a,N}} |\omega_k|^2\bigg)^{\frac{1}{2}} \sup\limits_{k\in\G_{\a,N}}\|u\|_{{H^2}(\square_k)},
\end{align*}
hold true, where $C$ is a constant independent of $\e$, $u$, $N$, $\omega_k$.
For $i=3$, in the above estimate we assume $\pL_3=\pL_3(\e \omega_k)$.
\end{lemma}

\begin{proof}
Due to the definition of $\pL_i$ and Cauchy-Schwarz inequality we have
\begin{align*}
\bigg|\Big(\sum\limits_{k\in\G_{\a,N}}\e \omega_k \S(-k)\pL_i \S(k) u,\psi_0\Big)_{L_2(\Pi_{\a,N})}\bigg|=&\bigg|\sum\limits_{k\in\G_{\a,N}}\e \omega_k \Big(\S(-k)\pL_i \S(k) u,\psi_0\Big)_{L_2(\square_k)}\bigg|
\\
=\bigg|\sum\limits_{k\in\G_{\a,N}}\e \omega_k \Big(\pL_i \S(k) u,\psi_0\Big)_{L_2(\square)}\bigg|
\leqslant &
\sum\limits_{k\in\G_{\a,N}} \e |\omega_k| \|\pL_i\S(k)u\|_{{L_2}(\square)} \|\psi_0\|_{L_2(\square)}.
\end{align*}
Since $\pL_i\colon H^2(\square)\to L_2(\square)$ is bounded, we find some constant $C$ such that
\begin{align*}
\sum\limits_{k\in\G_{\a,N}} & \e |\omega_k| \|\pL_i\S(k)u\|_{{L_2}(\square)} \|\psi_0\|_{L_2(\square)}
\leqslant   C \e \sum\limits_{k\in\G_{\a,N}}  |\omega_k| \|\S(k)u\|_{{H^2}(\square)}
\\
&\leqslant  C \e \bigg(\sum\limits_{k\in\G_{\a,N}} \omega_k^2 \bigg)^{\frac{1}{2}}
\left(\sum\limits_{k\in\G_{\a,N}} \|\S(k)u\|_{{H^2}(\square)}^2 \right)^{\frac{1}{2}}= C \e \bigg(\sum\limits_{k\in\G_{\a,N}} \omega_k^2 \bigg)^{\frac{1}{2}}
\|u\|_{{H^2}(\Pi_{\a,N})},
\end{align*}
and we arrive at the first desired estimate. The proof of the other is similar:
\begin{align*}
\bigg\|\sum\limits_{k\in\G_{\a,N}}\e \omega_k \S(-k)\pL_i \S(k) u\bigg\|_{L_2(\Pi_{\a,N})}^2= &\sum\limits_{k\in\G_{\a,N}}\e^2 \omega_k^2 \|\S(-k)\pL_i \S(k) u\|_{L_2(\square_k)}^2
\\
= \sum\limits_{k\in\G_{\a,N}} \e^2 \omega_k^2 \|\pL_i\S(k)u\|_{L_2(\square)}^2
\leqslant & C \e^2 \sup\limits_{k\in\G_{\a,N}}\|\S(k)u\|_{{H^2}(\square)}^2 \sum\limits_{k\in\G_{\a,N}} \omega_k^2.
\end{align*}
\end{proof}

This lemma, Lemma~\ref{lm3.2} and the properties of operators $\pL_i$ allows us to estimate two terms in the right hand side of (\ref{3.10}). Namely, we have
\begin{equation}\label{3.11}
\begin{aligned}
\frac{1}{N^n}&\Big|
\big(\pL_{\a,N}^\e
\Rs_{\a,N}(\L_0) \Rs_{\a,N}(\l_{\e,N}^\a)
\pL_{\a,N}^\e\psi_0,\psi_0\big)_{L_2(\Pi_{\a,N})}\Big|
\\
&\leqslant \frac{C\e}{N^n}  \bigg(\sum\limits_{k\in\G_{\a,N}} \omega_k^2 \bigg)^{\frac{1}{2}} \Big\| \Rs_{\a,N}(\L_0) \Rs_{\a,N}(\l_{\e,N}^\a)
\pL_{\a,N}^\e\psi_0 \Big\|_{{H^2}(\Pi_{\a,N})}
\\
&
\leqslant  \frac{C\e^2 N^4}{N^n} \|\psi_0\|_{{H^2}(\Pi_{\a,N})}  \bigg(\sum\limits_{k\in\G_{\a,N}} \omega_k^2 \bigg)^{\frac{1}{2}}\leqslant \frac{C c_{1}}{N^{\frac{n}{2}+4}} \|\psi_0\|_{{H^2}(\Pi_{\a,N})}\leqslant \frac{Cc_1}{N^4}\leqslant \frac{1}{2}
\end{aligned}
\end{equation}
provided $N_{1}$ in (\ref{3.1}) is great enough. In the same way we get
\begin{equation}\label{3.12}
\begin{aligned}
&\Big|\Big(\big(\pL_{\a,N}^\e \Rs_{\a,N}(\l_{\a,N}^\e)\big)^2(\I+\pL_{\a,N}^\e \Rs_{\a,N}(\l_{\a,N}^\e))^{-1}\pL_{\a,N}^\e\psi_0,\psi_0\Big)_{L_2(\Pi_{\a,N})}\Big|
\\
&\leqslant C\e\bigg(\sum\limits_{k\in\G_{\a,N}} \omega_k^2\bigg)^{\frac{1}{2}} \big\| \Rs_{\a,N}(\l_{\a,N}^\e) \pL_{\a,N}^\e \Rs_{\a,N}(\l_{\a,N}^\e) (\I+\pL_{\a,N}^\e \Rs_{\a,N}(\l_{\a,N}^\e))^{-1}\pL_{\a,N}^\e\psi_0
\big\|_{H^2(\Pi_{\a,N})}
\\
&\leqslant C\e^2 N^4 \bigg(\sum\limits_{k\in\G_{\a,N}} \omega_k^2\bigg)^{\frac{1}{2}} \|\pL_{\a,N}^\e\psi_0\|_{L_2(\Pi_{\a,N})}
\\
&= C\e^2 N^4 \bigg(\sum\limits_{k\in\G_{\a,N}} \omega_k^2\bigg)^{\frac{1}{2}} \bigg(\sum_{k \in \Gamma_{\a,N}}\| \S(k) \pL(\e \omega_k) \S(-k)\psi_0\|^2_{L_2(\square_k)}\bigg)^{\frac{1}{2}}
\\
&\leqslant C\e^2 N^4  \bigg(\sum\limits_{k\in\G_{\a,N}} \omega_k^2\bigg)^{\frac{1}{2}} \bigg(\sum_{k \in \Gamma_{\a,N}}\e^2 |\omega_k|^2\|\psi_0\|^2_{H^2(\square)}\bigg)^{\frac{1}{2}}
\leqslant C\e^3 N^4 \sum\limits_{k\in\G_{\a,N}} \omega_k^2.
\end{aligned}
\end{equation}
The term $N^4$ in the third line comes about due to  Lemma \ref{lm3.1} and estimate (\ref{3.7}).

By (\ref{2.1}) we can rewrite two other terms in the right hand side of (\ref{3.10}) as follows:
\begin{equation}\label{3.13}
(\pL_{\a,N}^\e\psi_0,\psi_0)_{L_2(\Pi_{\a,N})} -
\big(\pL_{\a,N}^\e
\Rs_{\a,N}(\L_0)\pL_{\a,N}^\e\psi_0,\psi_0\big)_{L_2(\Pi_{\a,N})}
=J_1+J_2+J_3,
\end{equation}
where
\begin{align*}
J_1=&\e\sum\limits_{k\in\G_{\a,N}} \omega_k(\S(k)\pL_1 \S(-k)\psi_0,\psi_0)_{L_2(\Pi_{\a,N})},
\\
J_2=&\e^2\sum\limits_{k\in\G_{\a,N}} \omega_k^2(\S(k)\pL_2 \S(-k)\psi_0,\psi_0)_{L_2(\Pi_{\a,N})}
\\
&-\e^2 \sum\limits_{p,k\in\G_{\a,N}} \omega_k \omega_p\big( \S(k) \pL_1 \S(-k) \Rs_{\a,N}(\L_0) \S(p)\pL_1 \S(-p)\psi_0,\psi_0\big)_{L_2(\Pi_{\a,N})},
\\
J_3=&\e^3\sum\limits_{k\in\G_{\a,N}} \omega_k^3 (\S(k)\pL_3(\e \omega_k) \S(-k)\psi_0,\psi_0)_{L_2(\Pi_{\a,N})}
\\
&-\e^2\sum\limits_{k\in\G_{\a,N}} \omega_k^2  \bigg( \S(k) \widehat{\pL}(\e \omega_k) \S(-k) \Rs_{\a,N}(\L_0) \pL_{\a,N}^\e(\omega) \psi_0,\psi_0\bigg)_{L_2(\Pi_{\a,N})}
\\
&-\e^2\sum\limits_{k\in\G_{\a,N}} \omega_k^2 \bigg(\pL_{\a,N}^\e(\omega) \Rs_{\a,N}(\L_0) \S(k) \widehat{\pL}(\e \omega_k) \S(-k) \psi_0,\psi_0\bigg)_{L_2(\Pi_{\a,N})},
\end{align*}
where $\widehat{\pL}(t):= \pL_2+ t  \pL_3(t)$.

Employing Lemmata~\ref{lm3.2},~\ref{lm3.3} and the properties of $\pL_i$, we estimate $J_3$:
\begin{equation}\label{3.14}
\begin{aligned}
|J_3|\leqslant &C\e^3\sum\limits_{k\in\G_{\a,N}} \omega_k^2
+C\e^2\bigg(\sum\limits_{k\in G_{\a,N}}\omega_k^2\bigg)^{\frac{1}{2}}
\\
&+ C \e^3 \bigg(\sum\limits_{k\in G_{\a,N}}\omega_k^2\bigg)^{\frac{1}{2}} \Big\|\Rs_{\a,N}(\L_0)
\sum\limits_{k\in\G_{\a,N}} \S(k) \widehat{\pL}(\e \omega_k) \S(-k)\psi_0\Big\|_{{H^2}(\Pi_{\a,N})}
\\
\leqslant & C\e^3 N^2\sum\limits_{k\in\G_{\a,N}} \omega_k^2.
\end{aligned}
\end{equation}

By assumption (\ref{as1}) term $J_1$ vanishes:
\begin{equation}\label{3.15}
\begin{aligned}
J_1=\e\sum\limits_{k\in\G_{\a,N}}\omega_k
\big(\pL_1\S(-k)\psi_0,\S(-k)\psi_0\big)_{L_2(\Pi_{\a,N})} =\e\sum\limits_{k\in\G_{\a,N}}\omega_k\big(\pL_1\psi_0, \psi_0\big)_{L_2(\Pi_{\square})}=0.
\end{aligned}
\end{equation}

To estimate $J_2$, we shall make use of one more auxiliary lemma.

\begin{lemma}\label{lm3.4}
The estimate
\begin{equation}
J_2\geqslant \e^2 c_{0} \sum\limits_{k\in\G_{\a,N}} \omega_k^2
\end{equation}
holds true, where $c_{0}$ is defined in assumption (\ref{as2}).
\end{lemma}

\begin{proof}
We denote
\begin{equation*}
f:=\sum\limits_{p\in\G_{\a,N}} \omega_p \S(p) \pL_1 \S(-p)\psi_0,\quad u_\l:=\Rs(\l)f
\end{equation*}
for $\l$ in a small neighborhood of $\L_0$. By (\ref{3.15}), function $f$ is orthogonal to $\psi_0$ in $L_2(\Pi_{\a,N})$. In view of this fact and by the definition of reduced resolvent $\Rs$, it is easy to make sure that $u_\l$ solves the equation
\begin{equation*}
(\Op_\square-\l) u_\l=f.
\end{equation*}
Function $u_\l$ is orthogonal to $\psi_0$ by the definition of $\Rs$.
Due to the symmetricity of $\pL_1$ we have
\begin{equation}\label{3.20}
J_2=\e^2\sum\limits_{k\in\G_{\a,N}} \omega_k^2(\S(k)\pL_2 \S(-k)\psi_0,\psi_0)_{L_2(\Pi_{\a,N})}
-\e^2 (u_{\L_0},f)_{L_2(\Pi_{\a,N})}.
\end{equation}

The main idea of this proof is to employ the variational formulation of the boundary value problem for $u_\l$. Namely,
given a domain $\Om$ and a part $\g$ of its boundary, by $\Ho^1(\Om,\g)$ we denote the subspace of $H^1(\Om)$ formed by functions vanishing on $\g$. As $\l\not=\L_0$, function $u_\l$ minimizes the functional
\begin{equation*}
\fB(u):=\|\nabla u\|_{L_2(\Pi_{\a,N})}^2 + (V_0U,U)_{L_2(\Pi_{\a,N})} -\L_0\|u\|_{L_2(\Pi_{\a,N})}^2 - 2(f,u)_{L_2(\Pi_{\a,N})}
\end{equation*}
over $\Ho^1(\Om,\g_{\a,N})$ and
\begin{equation*}
-(u_\l,f)_{L_2(\Pi_{\a,N})}=\fB(u_\l).
\end{equation*}

As in the proof of Lemma~\ref{lm3.3}, we introduce additional Neumann condition on the lateral boundaries $\p\square_k\setminus\p\Pi$ and it allows to estimate $\fB(u_\l)$ from below. More precisely, by $W:=\bigoplus\limits_{k\in\G_{\a,N}} \Ho^1(\square_k,\p\square_k\cap\p\Pi)$ we denote the subspace of $L_2(\Pi_{\a,N})$  consisting of the functions such that their restriction on $\square_k$ belongs to $\Ho^1(\square_k,\p\square_k\cap\p\Pi)$ for each $k\in\G_{\a,N}$. It is clear that $\Ho^1(\Pi_{\a,N},\g_{\a,N})\subset W$ and thus,
\begin{equation}\label{3.21}
\begin{aligned}
-&(u_\l,f)_{L_2(\Pi_{\a,N})}=\fB(u_\l)
\\
&\geqslant \inf\limits_{W} \sum\limits_{k\in\G_{\a,N}} \big(\|\nabla u\|_{L_2(\square_k)}^2
+ (V_0U,U)_{L_2(\square_k)}  -\L_0 \|u\|_{L_2(\square_k)}^2-2(f,u)_{L_2(\square_k)}\big).
\end{aligned}
\end{equation}
The functional in the right hand side of the above inequality is minimized by the solution to the equation
\begin{equation*}
\left(-\D+V_0-\L_0\right) v_\l=f\quad \text{in}\quad \Pi_{\a,N}\setminus\bigcup\limits_{k\in\G_{\a,N}} \p\square_k,
\end{equation*}
subject to boundary condition (\ref{2.8}) on $\bigcup\limits_{k\in \G_{\a,N}} \p\square_k\cap\p\Pi$ and to Neumann condition on $\bigcup\limits_{k\in \G_{\a,N}} \p\square_k\setminus\p\Pi$. This solutions reads as $v_\l=\omega_k \Rs_{k,1}(\l) \S(k)\pL_1\psi_0$ on $\square_k$, $k\in\G_{\a,N}$. The restriction of $v_\l$ on $\square_k$ is orthogonal to $\psi_0$ in $L_2(\square_k)$ for each $k\in\G_{\a,N}$.
Hence, inequality (\ref{3.21}) takes the form of
\begin{equation}\label{3.22}
-(u_\l,f)_{L_2(\Pi_{\a,N})}\geqslant
-(v_\l,f)_{L_2(\Pi_{\a,N})}.
\end{equation}
It is clear that $u_\l$ and $v_\l$ are continuous w.r.t. $\l$ in $L_2(\Pi_{\a,N})$ and $v_{\L_0}=\S(k)U$ on $\square_k$, where, we remind, function $U$ was introduced in Assumption~(\ref{as2}). The latter identity and (\ref{3.22}) yield
\begin{align*}
-\Big(u_{\L_0},f\Big)_{L_2(\Pi_{\a,N})}\geqslant &
-\Big(v_{\L_0},f\Big)_{L_2(\Pi_{\a,N})}
= -\sum\limits_{k\in\G_{\a,N}} \omega_k^2 \big(\S(k)U,\S(k)\pL_1\S(-k)\psi_0\big)_{L_2(\square_k)}
\\
=& -\sum\limits_{k\in\G_{\a,N}} \omega_k^2 (U,\pL_1\psi_0)_{L_2(\square_k)} = -(U,\pL_1\psi_0)_{L_2(\square_k)} \sum\limits_{k\in\G_{\a,N}} \omega_k^2.
\end{align*}
We also have
\begin{align*}
\sum\limits_{k\in\G_{\a,N}} \omega_k^2(\S(k)\pL_2 \S(-k)\psi_0,\psi_0)_{L_2(\Pi_{\a,N})} = &
\sum\limits_{k\in\G_{\a,N}} \omega_k^2(\pL_2 \S(-k)\psi_0,\S(-k)\psi_0)_{L_2(\square_k)}
\\
=& (\pL_2 \psi_0,\psi_0)_{L_2(\square)}
 \sum\limits_{k\in\G_{\a,N}} \omega_k^2.
\end{align*}
Now the lemma follows from last two estimates, (\ref{3.20}) and Assumption~(\ref{as2}).
\end{proof}

We return back to identity (\ref{3.10}). First it follows from (\ref{3.11}) and (\ref{3.1}) that
\begin{equation*}
1+\frac{1}{N^n}
\big(\pL_{\a,N}^\e
\Rs_{\a,N}(\L_0) \Rs_{\a,N}(\l_{\e,N}^\a)
\pL_{\a,N}^\e\psi_0,\psi_0\big)_{L_2(\Pi_{\a,N})}\geqslant \frac{1}{2}.
\end{equation*}
This inequality, (\ref{3.12}), (\ref{3.13}), (\ref{3.14}), and Lemma~\ref{lm3.4} allow us to estimate the right hand side of (\ref{3.10}) and to obtain in this way the estimate for the left hand side:
\begin{align*}
\l_{\a,N}^\e-\L_0\geqslant \frac{1}{N^n} \left( 2c_{0}\e^2 \sum\limits_{k\in\G_{\a,N}} \omega_k^2 - C\e^3N^4 \sum\limits_{k\in\G_{\a,N}} \omega_k^2 \right) \geqslant \frac{2c_{0}\e^2}{N^n}\left(1-\frac{C\e N^4}{2c_{0}}\right) \sum\limits_{k\in\G_{\a,N}} \omega_k^2.
\end{align*}
By (\ref{3.1}) it completes the proof of Theorem~\ref{th3.1} provided $c_{1}$ is small enough.

\section{Combes-Thomas estimates}
\label{s:CTe}

We establish a Combes-Thomas estimate for the class of operators introduced in Section \ref{s:results}.
We use that they are block diagonal with respect to the decomposition $\bigoplus\limits_{k\in\G} L_2\big(\square+(k,0)\big)$. They need not be differential operators. For the reader's convenience we formulate here Theorem~\ref{th2.2} once again.

\begin{theorem}[Theorem~\ref{th2.2} above.]
Let $\a, \b_1, \b_2\in G$, $m_1, m_2\in\NN$ be such that $B_1:=\Pi_{\b_1,m_1}\subset \Pi_{\a,N}$, $B_2:=\Pi_{\b_2,m_2}\subset \Pi_{\a,N}$. There exists $N_2\in \NN$ such that for $N\geqslant N_2$ the bound
\begin{equation}
\|\chi_{B_1}(\Op^\e_{\a,N}(\omega)-\l)^{-1}\chi_{B_2}\|_{L_2(\Pi_{\a,N})\to L_2(\Pi_{\a,N})} \leqslant \frac{C_1}{\d} \E^{-C_2\d \dist(B_1, B_2)},
\end{equation}
holds, where $C_1$, $C_2$ are positive constants independent of $\e$, $\a$, $N$, $\d$, $\b_1$, $\b_2$, $m_1$, $m_2$, $\l$ and  $\d:=\dist(\l,\spec(\Op^\e_{\a,N}(\omega)))>0$.

\end{theorem}

\begin{proof}
We fix $\om\in\Om_N$. For arbitrary $M$ we introduce the function $J=J(t,M)$ on $[0,M]$:
\begin{equation*}
J(t,M)=t-\z(t)+\z(M-t),
\end{equation*}
where $\z\in C^\infty[0,+\infty)$, $\z(t)=0$ outside $[0,1]$, $\z'(0)=1$.

Let $e_j^\bot$, $j=1,\ldots,n$, be the basis in $\RR^n$ determined by the conditions
\begin{equation}\label{4.6}
(e_i,e_j^\bot)_{\RR^n}=\pm \d_{ij},
\end{equation}
where $\d_{ij}$ is the Kronecker delta. The sign in the above conditions is chosen by the following rule. Given $x\in\Pi_{\a,N}$, we represent $x'$ as
\begin{equation}\label{4.7}
x'=\a+\sum\limits_{j=1}^{n} b_j e_j^\bot,
\end{equation}
and $b_j$ belong to the segments $[0,M_j N]$, where $M_j>0$ are some constants. The latter condition on positivity of $M_j$ determines uniquely the signs in (\ref{4.6}) and consequently, vectors $e_j^\bot$. We also observe that $M_j N$ are in fact the distance between the opposite lateral sides of the parallelepiped $\Pi_{\a,N}$.

 By means of expansion (\ref{4.7}) and function $J$ we define one more function on $\Pi_{\a,N}$:
\begin{equation*}
J_*(x')=\sum\limits_{j=1}^{n} J(b_j,M_j N).
\end{equation*}
It is straightforward to check that the gradient of function $J_*$ vanishes on the lateral boundaries of $\Pi_{\a,N}$
and $J_*\in C^\infty(\overline{\Pi_{\a,N}})$ provided $N\geqslant N_2$ and $N_2$ is large enough.

Given $a>0$, by $\cT_a$ we denote the  multiplication operator:
$\cT_a u:=\E^{a J_*} u$ in $L_2(\Pi_{\a,N})$. It follows from the aforementioned properties of $J_*$ that $\cT_a$ maps the domain of $\Op_{\a,N}^\e$ onto itself. It is also straightforward to check that
\begin{equation}\label{4.8}
\cT_{-a}\D \cT_a=\D+\cP^{(1)}(a),
\end{equation}
where $\cP^{(1)}(a)$ is a first order differential operator whose coefficients are bounded by $C(a+a^2)$ in $\Pi_{\a,N}$, where constant $C$ is independent of $x\in\Pi_{\a,N}$, $N$, $\a$.

The most important ingredient in the proof is obtaining an identity similar to (\ref{4.8}) for $\pL_{\a,N}^\e$. The first step follows from the definition of $\pL_{\a,N}^\e$:
\begin{equation}\label{4.9}
\begin{aligned}
\cT_{-a}\pL_{\a,N}^\e \cT_a =& \sum\limits_{k\in\G_{\a,N}} \E^{-a J_*} \S(k) \pL(\e \omega_k) \S(-k) \E^{a J_*}
\\
= &\sum\limits_{k\in\G_{\a,N}} \S(k) \E^{-a J_*(\cdot-k)}  \pL(\e \omega_k) \E^{a J_*(\cdot-k)}  \S(-k)
\\
=&\sum\limits_{k\in\G_{\a,N}} \S(k) \E^{-a \big(J_*(\cdot-k)-J_*(-k)\big)}  \pL(\e \omega_k) \E^{a \big(J_*(\cdot-k)-J_*(-k)\big)}  \S(-k).
\end{aligned}
\end{equation}
Then it is easy to prove the estimates
\begin{equation}\label{4.10}
\|J(\cdot-k)-J_*(-k)\|_{C_2(\overline{\square})}\leqslant C,
\end{equation}
where constant $C$ is independent of $k\in\G_{\a,N}$ and $N$. We also mention the inequalities
\begin{equation*}
\E^{t}-1\leqslant t\E^{t},\quad t\in[0,+\infty),\qquad 1-\E^{t}\leqslant -t,\quad t\in(-\infty,0],
\end{equation*}
which can be easily proven by checking the monotonicity of the functions $t\mapsto (t-1)\E^{t}+1$, $t\in[0,+\infty)$, $t\mapsto t-\E^{-t}+1$, $t\in[0,+\infty)$.
These inequalities and
(\ref{4.10}) allow us to bound
$\| \E^{-a \big(J_*(\cdot-k)-J_*(-k)\big)}  \pL(\e \omega_k) \E^{a \big(J_*(\cdot-k)-J_*(-k)\big)} -\pL(\e \omega_k)\|_{{H^2}(\Pi_{\a,N})\to L_2(\Pi_{\a,N})}$.
Using the expansion (\ref{4.9}) this yield the desired relations:
\begin{align*}
\|\cP^{(2)}(a,\e,\a,N)\|_{{H^2}(\Pi_{\a,N})\to L_2(\Pi_{\a,N})}
\leqslant C a\,\E^{2a},
\quad
\text{ for }\cP^{(2)}(a,\e,\a,N):=\cT_{-a} \pL^\e_{\a,N} \cT_a-\pL^\e_{\a,N},
\end{align*}
where $\|\cdot\|_{X\to Y}$ indicates the norm of an operator from a Hilbert space $X$ into a Hilbert space $Y$, and $C$ is a constant independent of $a$, $\e$, $\a$, $N$. This estimate and (\ref{4.8}) imply
\begin{equation}\label{4.15}
\cT_{-a} \Op^\e_{\a,N} \cT_{a} =\Op^\e_{\a,N} + \cP(a,\e,\a,N),
\end{equation}
where $\cP$ is a bounded operator from ${H^2}(\Pi_{\a,N})$ into $L_2(\Pi_{\a,N})$ obeying the estimate
\begin{equation}\label{4.12}
\|\cP(a,\e,\a,N)\|_{{H^2}(\Pi_{\a,N})\to L_2(\Pi_{\a,N})}\leqslant C a \E^{2a}.
\end{equation}
Here $C$ is a constant independent of $a$, $\e$, $\a$, $N$.
This estimate and the previous identity (\ref{4.15}) is the key idea in the proof.
It is exactly these two ingredients which allow us to follow now the established strategy of the proof of a Combes-Thomas estimate, see e.g.{} Corollary~3.3 in \cite{BorisovV-11}.

Our next step is the estimate for the resolvent of $\Op^\e_{\a,N}$. We assume that $\l\in[\L_0,\L_0+1]$ and provided $\l$ belongs to the resolvent set of $\Op^\e_{\a,N}$, we have
\begin{equation*}
\|(\Op^\e_{\a,N}-\l)^{-1}\|=\frac{1}{\d},\quad
\d:=\dist(\l,\spec(\Op^\epsilon_{\a,N})).
\end{equation*}
This identity and the obvious ones
\begin{align*}
&\Op^\e_{\a,N}-\l=\Op^0_{\a,N}+\iu+\pL^\e_{\a,N}-\iu-\l,
\\
&(\Op^\epsilon_{\a,N}-\l)^{-1}=(\Op^0_{\a,N}+\iu)^{-1}\big(\I +(\pL^\e_{\a,N}-\l-\iu)^{-1}(\Op^0_{\a,N}+\iu)^{-1}\big)^{-1}
\end{align*}
yield
\begin{equation*}
\|(\Op^\e_{\a,N}-\l)^{-1}\|_{L_2(\Pi_{\a,N})\to {H^2}(\Pi_{\a,N})} \leqslant \frac{C}{\dist(\l,\spec(\Op^\e_{\a,N}))},
\end{equation*}
where constant $C$ is independent of $\e$, $\a$, $N$, and $\l$. This estimate and (\ref{4.15}), (\ref{4.12}) imply  the inequality
\begin{equation*}
\|\cP(a,\e,\a,N)(\Op^\e_{\a,N}-\l)^{-1}\|\leqslant \frac{C a\, \E^{2a}}{\d},
\end{equation*}
where constant  $C$ is independent of $a$, $\e$, $\a$, $N$, and $\d$. Hence, for $a=C\d$ with a sufficiently small $C$,
\begin{equation}\label{4.17}
\big\|\big(\Op^\e_{\a,N}+\cP(a,\e,\a,N)-\l\big)^{-1}\big\|\leqslant \frac{C}{\d},
\end{equation}
where constant  $C$ is independent of $a$, $\e$, $\a$, $N$, and $\d$.

Given $\b_1, \b_2\in\G_{\a,N}$, $m_1, m_2>0$ such that $\Pi_{\b_1,m_1}\subset \Pi_{\a,N}$, $\Pi_{\b_2,m_2}\subset \Pi_{\a,N}$,
by (\ref{4.17}) for each normalized vectors $\psi_1, \psi_2\in L_2(\Pi_{\a,N})$ we have
\begin{equation}\label{4.18}
\begin{aligned}
\Big|\big(|\psi_1|\chi_{B_1}, \cT_{-a}(\Op^\epsilon_{\a,N}-\l)^{-1} \cT_a \chi_{B_2}|\psi_2|\big)_{L_2(\Pi_{\a,N})}\Big| \leqslant &\|\cT_{-a} (\Op^\e_{\a,N}-\l)^{-1} \cT_a\|
\\
=& \big\|\big(\Op^\e_{\a,N}+\cP(a,\e,\a,N)-\l\big)^{-1}
\big\| \leqslant \frac{C}{\d},
\end{aligned}
\end{equation}
where constant $C$ is independent of $\d$, $\e$, $\a$, $N$, and $\l$, $a$ is chosen as indicated above, and we remind that $B_1=\Pi_{\b_1,m_1}$, $B_2=\Pi_{\b_2,m_2}$. Since $\l$ is below the spectrum of $\Op^\e_{\a,N}$, the integral kernel of $(\Op^\e_{\a,N}-\l)^{-1}$ is positive. Without loss of generality we assume that $|\b_2|\geqslant |\b_1|$, the opposite case is studied in the same way. From now we assume $ \supp \psi_j \subset B_j$ $(j=1,2)$.
Then it is straightforward to check that
\begin{align*}
\Big|\big(|\psi_1|\chi_{B_1}, \cT_{-a}(\Op^\epsilon_{\a,N}-\l)^{-1} \cT_a \chi_{B_2}|\psi_2|\big)_{L_2(\Pi_{\a,N})}\Big| \geqslant & \exp\left(\left(\min\limits_{B_2} J_* - \max
\limits_{B_1} J_*\right) C\d\right)
\\
&\cdot (\psi_1,(\Op^\e_{\a,N}-\l)^{-1}\psi_2)_{L_2(\Pi_{\a,N})}.
\end{align*}
And since
\begin{equation*}
\min\limits_{B_2} J_* - \max \limits_{B_1} J_* \geqslant C\dist(B_1, B_2),
\end{equation*}
where $C$ is a positive constant independent of $\b_1$, $\b_2$, $m_1$, $m_2$, $\a$, $N$, two latter inequalities and (\ref{4.18}) imply
\begin{equation*}
\big|(\psi_1,(\Op^\e_{\a,N}-\l)^{-1}\psi_2)_{L_2(\Pi_{\a,N})}\big| \leqslant \frac{C_1}{\d} \E^{-C_2\d \dist(B_1, B_2)},
\end{equation*}
where $C_1$, $C_2$ are positive constants independent of $\e$, $\a$, $N$, $\d$, $\b_1$, $\b_2$, $m_1$, $m_2$.
\end{proof}
\section{Probabilistic estimates}
\label{s:probabilistic.estimates}

In this section we prove our two probabilistic results, Theorem~\ref{th2.1} and Theorem~\ref{th2.2}.

\begin{proof}[Proof of Theorem~\ref{th2.1}] We follow the main lines of the proof of Theorem~3.1. in \cite{BorisovV-11}. We choose $K,\g\in \NN$ and we let $N:=K^\g$. Then up to a set of measure zero we can partition $\Pi_{\a,N}$ into smaller pieces $\Pi_{\b,K}$:
\begin{equation*}
\Pi_{\a,N}=\bigcup\limits_{\b\in M_{K,\g}}^{\bullet} \Pi_{\b,K},
\end{equation*}
where $\bigcup\limits^{\bullet}$ stands for the disjoint union and $M_{K,\g}$ is the set
$M_{K,\g}=K\G\cap\G_{\a,N}$.
We observe that the number of elements in the set $M_{K,\g}$ is equal to $(N/K)^n=(K^{\g-1})^n=N^{n\left(1-\frac{1}{\g}\right)}$. On the lateral boundaries of $\Pi_{\b,K}$ we impose Neumann boundary condition and by the minimax principle we obtain
\begin{equation*}
\l_{\a,N}^\e\geqslant \min\limits_{\b\in M_{K,\g}} \l_{\b,K}^\e.
\end{equation*}
Let us reformulate the above estimate in probabilistic terms. First it implies that
\begin{equation}\label{4.5}
\big\{\omega\in\Omega:\, \l_{\a,N}^\e-\L_0\leqslant N^{-\frac{1}{2}}
\big\}\subseteq\bigcup\limits_{\b\in M_{K,\g}}\big\{\omega\in\Omega:\,
\l_{\b,K}^\e-\L_0\leqslant K^{-\frac{\g}{2}}\big\}
\end{equation}
Since random variables $\omega_k$, $k\in\G$ are independent and identically distributed,
\begin{equation}\label{4.4}
\sum\limits_{\b\in M_{K,\g}} \PP\left(\omega\in\Om:\, \l_{\b,K}^\e-\L_0\leqslant K^{-\frac{\g}{2}}\right) \leqslant N^{n\left(1-\frac{1}{\g}\right)} \PP\left(\omega\in\Om:\, \l_{\a,K}^\e-\L_0\leqslant K^{-\frac{\g}{2}}\right).
\end{equation}

The Cauchy-Schwarz inequality
\begin{equation*}
\frac{1}{K^{\frac{n}{2}}}\sum\limits_{k\in\G_{\a,K}} |\omega_k|\leqslant \Big(\sum\limits_{k\in\G_{\a,K}} |\omega_k|^2\Big)^{\frac{1}{2}}
\end{equation*}
and Theorem~\ref{th3.1} yield for
$K \geqslant N_1$ and $\e\leqslant c_{1}K^{-4}$
\begin{equation}\label{4.2}
\begin{aligned}
&\big\{\omega\in\Omega:\, \l_{\a,K}^\e-\L_0\leqslant K^{-\frac{\g}{2}}
\big\}\subseteq\left\{\omega\in\Omega:\, \frac{c_{2}\e^2}{K^n}\sum\limits_{k\in\G_{\a,K}} \omega_k^2\leqslant K^{-\frac{\g}{2}}\right\}
\\
&=\left\{\omega\in\Omega:\, \Big(\sum\limits_{k\in\G_{\a,K}} \omega_k^2\Big)^{\frac{1}{2}} \leqslant \frac{K^{\frac{n}{2}-\frac{\g}{4}}}{\sqrt{c_{2}}\e}\right\}
\subseteq\left\{\omega\in\Omega:\, \frac{1}{K^{\frac{n}{2}}} \sum\limits_{k\in\G_{\a,K}} |\omega_k| \leqslant \frac{K^{\frac{n}{2}-\frac{\g}{4}}}{\sqrt{c_{2}}\e}
\right\}
\\
&=\left\{ \omega\in\Omega:\, \frac{1}{K^n} \sum\limits_{k\in\G_{\a,K}} |\omega_k| \leqslant \frac{K^{-\frac{\g}{4}}}{\sqrt{c_{2}}\e}
\right\}.
\end{aligned}
\end{equation}
We choose $\e$ so that
\begin{equation}\label{4.3}
\frac{K^{-\frac{\g}{4}}}{\sqrt{c_{2}}\e}
\leqslant \frac{\EE(|\omega_k|)}{2},
\end{equation}
i.e.,
\begin{equation*}
\e\geqslant \frac{2}{\sqrt{c_{2}}\,\EE(|\omega_k|)}           K^{-\frac{\g}{4}}.
\end{equation*}
It is clear that this inequality is compatible with (\ref{3.1}) provided $K\geqslant K_1$,
where $K_1$ is large enough, depending only on $c_{1}, c_{2},\gamma, n$, and $\EE(|\omega_0|)$.
We apply the large deviation principle analogously as in \cite[Lm. 4.3]{BorisovV-11}). Hence there exists a constant $c_{4}>0$ depending on $\mu$ only such that for each $K\in\NN$
\begin{equation*}
\PP\left(\omega\in\Om:\, \frac{1}{K^n}\sum\limits_{k\in\G_{\a,K}} |\omega_k| \leqslant \frac{\EE(|\omega_0|)}{2}\right)\leqslant \E^{-c_{4}K^n}.
\end{equation*}
Thus, by (\ref{4.2}), (\ref{4.3}) it follows that
\begin{align*}
\PP\left(\omega\in\Om:\, \frac{1}{K^n}\sum\limits_{k\in\G_{\a,K}} |\omega_k|\leqslant \frac{K^{\frac{\g}{4}}}{\sqrt{c_{2}}\e}\right) \leqslant \PP\left(\omega\in\Om:\, \frac{1}{K^n}\sum\limits_{k\in\G_{\a,K}} |\omega_k|\leqslant  \frac{\EE(|\omega_k|)}{2}\right)\leqslant \E^{-c_{4}K^n}
\end{align*}
as soon as  $K\geqslant K_1$. Therefore, provided $N\geqslant \max\{N_1^\gamma, K_1^\gamma\}$, where $N_1$ comes from Theorem~\ref{th3.1}, by (\ref{4.4}), (\ref{4.5}) we get
\begin{equation*}
\PP\left(\omega\in\Om:\, \l_{\a,K}^\e-\L_0\leqslant N^{-\frac{1}{2}} \right)\leqslant N^{n\left(1-\frac{1}{\g}\right)} \E^{-c_{4}K^n}\leqslant  N^{n\left(1-\frac{1}{\g}\right)} \E^{-c_{4}N^{\frac{n}{\g}}}
\end{equation*}
completing the proof.
\end{proof}

\begin{proof}[Proof of Theorem~\ref{th2.2}] Here the main ideas are borrowed from the proof of Corollary~3.3 in \cite{BorisovV-11}. We first introduce the set
\begin{align*}
\Om_N:=&\left\{\omega\in\Omega:\, \l_{\a,N}^\e(N)-\L_0>\frac{1}{\sqrt{N}}
\right\}
=\left\{\omega\in\Omega:\, \dist(\L_0,\spec(\Op^\e_{\a,N}))>\frac{1}{\sqrt{N}}
\right\}
\\
=&\left\{\omega\in\Omega:\, \dist(\l,\spec(\Op^\e_{\a,N}))>\frac{1}{2\sqrt{N}},\, \forall \l\in\left[\L_0,\L_0+\frac{1}{2\sqrt{N}}\right]
\right\}.
\end{align*}
For the next step we need a Combes-Thomas estimate as it is given in Section \ref{s:CTe}.
We apply it to Hamiltonians $\Op^\e_{\a,N}(\omega)$ with configuration $\om$ in the set $\Om_N$.
As before we consider
$B_1:=\Pi_{\b_1,m_1}\subset\Pi_{\a,N}$, $B_2:=\Pi_{\b_2,m_2}\subset\Pi_{\a,N}$
and $\psi_j \in L_2(\Pi_{\a,N})$ with $ \supp \psi_j \subset B_j$ $(j=1,2)$.
The Combes-Thomas estimate implies:
\begin{equation*}
\big|(\psi_1,(\Op^\e_{\a,N}(\omega)-\l)^{-1}\psi_2)_{L_2(\Pi_{\a,N})}\big| \leqslant \frac{C_1}{\d} \E^{-C_2\d \dist(B_1, B_2)},
\end{equation*}
where $C_1$, $C_2$ are positive constants independent of $\e$, $\a$, $N$, $\d$, $\b_1$, $\b_2$, $m_1$, $m_2$, and $\d=\dist(\l,\spec(\Op^\e_{\a,N}(\omega)))$, $N\geqslant N_2$.


Now, fix $N\geqslant \max\{N_1^\gamma, K_1^\gamma, N_2\}$, $\omega\in\Om_N$ and $\l\in\left[\L_0,\L_0+\frac{1}{2\sqrt{N}}\right]$. Then $\d \geqslant \frac{1}{2\sqrt{N}}$ and thus
\begin{equation*}
\big|(\psi_1,(\Op^\e_{\a,N}-\l)^{-1}\psi_2)_{L_2(\Pi_{\a,N})}\big| \leqslant 2C_1\sqrt{N} \E^{-\frac{C_2}{2\sqrt{N}} \dist(B_1,B_2)}.
\end{equation*}
By Theorem~\ref{th2.1} we have bound
\begin{equation*}
\PP(\Om_N)\geqslant 1 - N^{n\left(1-\frac{1}{\g}\right)}\E^{-c_{4}N^{\frac{n}{\g}}},
\end{equation*}
and therefore,
\begin{align*}
\PP\Bigg(&\omega\in\Omega:\, \forall \l\in\left[\L_0,\L_0+\frac{1}{2\sqrt{N}}\right]\,
\\
& \big\|\chi_{B_1}(\Op^\e_{\a,N}-\l)^{-1}\chi_{B_2}\big\|\leqslant 2C_1
\sqrt{N} \E^{-\frac{C_2}{2\sqrt{N}} \dist(\Pi_{\b_1,m_1}, \Pi_{\b_2,m_2})}\Bigg)
\geqslant  1 - N^{n\left(1-\frac{1}{\g}\right)}\E^{-c_{4}N^{\frac{n}{\g}}}
\end{align*}
that completes the proof.
\end{proof}

\section*{Acknowledgments}

D.B. and A.G. were supported by Russian Science Foundation, project no. 14-11-00078. A.G. was supported in part by the Moebius Contest Foundation for Young Scientists. I.V. was supported financially by the DFG and the DAAD.

\def\polhk#1{\setbox0=\hbox{#1}{\ooalign{\hidewidth
  \lower1.5ex\hbox{`}\hidewidth\crcr\unhbox0}}}
\def\polhk#1{\setbox0=\hbox{#1}{\ooalign{\hidewidth
  \lower1.5ex\hbox{`}\hidewidth\crcr\unhbox0}}}

\end{document}